\documentclass[a4paper,reqno]{amsart}
\usepackage{amssymb}
\usepackage{amsmath,amssymb,mathrsfs,latexsym,tikz,cite,hyperref}
\usepackage{mathabx}
\usepackage{genyoungtabtikz}\usetikzlibrary{decorations.pathreplacing}
\usepackage{rotating, graphicx}
\usepackage{lscape,marginnote}\usepackage{mfabacus} 

\allowdisplaybreaks 

\title{Decomposition numbers for Rouquier blocks of Ariki-Koike algebras I}
\author[S.~Lyle]{Sin\'ead Lyle}
\address{School of Mathematics, University of East Anglia, Norwich NR4 7TJ, UK.}
\email{s.lyle@uea.ac.uk}

\numberwithin{equation}{section}
\numberwithin{figure}{section}
\newtheorem{lemma}{Lemma}[section]
\newtheorem{theorem}[lemma]{Theorem}
\newtheorem{proposition}[lemma]{Proposition}
\newtheorem{corollary}[lemma]{Corollary}
\newtheorem{thmx}{Theorem}

\newtheorem{conjx}[thmx]{Conjecture}

\theoremstyle{remark}

\newtheorem*{ex}{Example}

\newcommand{\bla}{\boldsymbol \la}
\newcommand{\bnu}{\boldsymbol \nu}
\newcommand{\bmu}{\boldsymbol \mu}
\newcommand{\bal}{\boldsymbol \alpha}
\newcommand{\bgam}{\boldsymbol \gamma}
\newcommand{\bdel}{\boldsymbol \delta}
\newcommand{\bDel}{\boldsymbol \Delta}
\newcommand{\bbe}{\boldsymbol \beta}
\newcommand{\ba}{\boldsymbol a}
\newcommand{\bb}{\boldsymbol b}
\newcommand{\bc}{\boldsymbol c}
\newcommand{\bd}{\boldsymbol d}
\newcommand{\bsig}{\boldsymbol \sigma}
\newcommand{\btau}{\boldsymbol \tau}
\newcommand{\bet}{\boldsymbol \eta}
\newcommand{\bep}{\boldsymbol \epsilon}
\newcommand{\al}{\alpha}
\newcommand{\be}{\beta}
\newcommand{\ep}{\epsilon}
\newcommand{\emp}{\varnothing}

\newcommand{\C}{\mathbb{C}}
\newcommand{\Z}{\mathbb{Z}}
\newcommand{\N}{\mathbb{N}}
\newcommand{\U}{\mathcal{U}}
\newcommand{\F}{\mathcal{F}}
\newcommand{\h}{\mathcal{H}}

\newcommand{\la}{\lambda}
\newcommand{\La}{\Lambda}
\newcommand{\Lar}{\Lambda^r_{\text{reg}}}
\newcommand{\Larn}{\Lambda^{r}_n}
\newcommand{\Kl}{\La^{\mc}_n}
\newcommand{\Kla}{\La^{\mc}}
\newcommand{\R}{\mathcal{R}}
\newcommand{\RR}{\mathcal{R}_{\text{reg}}}

\newcommand{\Rs}{\bar{\R}^s}
\newcommand{\Rrs}{\bar{\R}^s_{\text{reg}}}
\newcommand{\Rp}{\R^{\diamond}}

\newcommand{\mc}{\boldsymbol a}
\newcommand{\M}[2]{\Gamma^{#1}_{#2}}
\newcommand{\hook}{\mathrm{h}}
\newcommand{\ap}{\approx_{\mc}}
\newcommand{\gedom}{\trianglerighteq}
\newcommand{\gdom}{\triangleright}
\newcommand{\BB}{\mathfrak{B}}

\DeclareMathOperator{\res}{res}
\DeclareMathOperator{\Res}{Res}
\DeclareMathOperator{\qt}{qt}
\DeclareMathOperator{\add}{add}
\DeclareMathOperator{\rem}{rem}
\DeclareMathOperator{\End}{End}
\newcommand{\Gr}[1]{[\![ #1]\!]}

\begin{document}
\begin{abstract}
Let $\h=\h_{r,n}(q,{\bf Q})$ denote an Ariki-Koike algebra over a field of characteristic $p\geq 0$. For each $r$-multipartition $\bla$ of $n$, we define a $\h$-module $S^{\bla}$ and for each Kleshchev $r$-multipartition $\bmu$ of $n$, we define an irreducible $\h$-module $D^{\bmu}$. Given a multipartition $\bla$ and a Kleshchev multipartition $\bmu$ both lying in a Rouquier block such that $\bla$ and $\bmu$ have the same multicore, we give a closed formula for the graded decomposition number $[S^{\bla}:D^{\bmu}]_v$ when $p=0$ or when each component of $\bmu$ has fewer than $p$ removable $e$-rim hooks. 
\end{abstract}
\maketitle

\section{Introduction}
\medskip

Let $\h=\h_{r,n}(q,{\bf Q})$ denote an Ariki-Koike algebra. These algebras were introduced by Ariki and Koike~\cite{ArikiKoike} as a simultaneous generalization of the Hecke algebras of type $A$, when $r=1$ and type $B$, when $r=2$. In a natural generalization of the combinatorics which appear in the type $A$ case, there is a class of important $\h$-modules, called Specht modules, which are indexed by the set of $r$-multipartitions of $n$. When $\h$ is semisimple, these Specht modules form a complete set of non-isomorphic irreducible $\h$-modules; otherwise, the simple modules appear as the heads of a subset of the Specht modules.

One of the most important open problems in the representation theory of the Ariki-Koike algebras is to determine the multiplicity $[S^{\bla}:D^{\bmu}]$ of a simple module $D^{\bmu}$ as a composition factor of a Specht module $S^{\bla}$; it also seems to be a very difficult problem. Even when $r=1$, there are very few cases where there is a closed formula for these decomposition numbers. However when $r=1$, the decomposition numbers are known for certain blocks called Rouquier blocks or RoCK blocks. 

The Rouquier blocks for the Hecke algebras of type $A$ were defined by Rouquier~\cite{Rouquier}. Due to work on them by Chuang and Kessar~\cite{CK}, many authors refer to them as RoCK (Rouquier or Chuang-Kessar) blocks. 
Chuang and Kessar showed that a Rouquier block of $\mathbb{F}_p \mathfrak{S}_n$ of weight $w<p$ is Morita equivalent to the principal block of $\mathfrak{S}_p \wr \mathfrak{S}_n$, and hence they were able to prove Brou\'e's abelian defect group conjecture for the Rouquier blocks; Chuang and Rouquier~\cite{CR} later extended the proof of the conjecture to all blocks of $\mathbb{F}_p \mathfrak{S}_n$ by showing that any block of weight $w$ is derived equivalent to a Rouquier block. There is an elegant closed formula for the decomposition numbers of Specht modules lying in Rouquier blocks, proved by Leclerc and Miyachi~\cite{LeclercMiyachi2} when $p=0$; by Chuang and Tan~\cite{CT} for the symmetric group algebra when $p<w$; and by James, Mathas and the author~\cite{JLM} for a Hecke algebra over a field of characteristic $p<w$. Apart from the value of knowing decomposition numbers, this explicit formula has been used to study the other aspects of the Hecke algebras. Decomposition numbers for the Rouquier blocks were recently used in the proof that all blocks of of weight at least $2$ in quantum characteristic $e\geq 3$ are Schurian-infinite~\cite{ALS}; they appear in the classification of the irreducible Specht modules when $e>2$~\cite{Fayers:Irred}; they were used in the first construction of homomorphism spaces of dimension larger than $1$ between Specht modules when $e>2$~\cite{Dodge}; and they provide a starting point for many results about Specht modules in blocks of small weight~\cite{Fayers:W3,Fayers:W4}. 

In~\cite{L:Rouquier}, we introduced the notion of Rouquier blocks for the Ariki-Koike algebras, as a natural generalization of the Rouquier blocks for the Hecke algebras of type $A$. The description of Rouquier blocks we give is in terms of abacus combinatorics. Analogously to~\cite{Webster}, we use the term RoCK block to refer to a block which is Scopes equivalent to a Rouquier block; although our main theorem is stated (and proved) for Rouquier blocks, similar results hold for any RoCK block.   
Given the known results for $r=1$, it is natural to ask whether it is possible to have a closed formula for the decomposition numbers for the Rouquier blocks of the Ariki-Koike algebras. This paper gives a partial answer. When $\h$ is defined over a field of characteristic $0$ we give such a formula for the graded multiplicity $[S^{\bla}:D^{\bmu}]_v$ where $\bla$ and $\bmu$ lie in a Rouquier block and have a common multicore. The formula also holds when the characteristic of the field is larger than the number of removable $e$-rim hooks in every component of $\bmu$. Additional motivation for our results is given by work of Muth, Speyer and Sutton~\cite{MSS} who show that these decomposition numbers should be related to decomposition numbers for cell modules in the cyclotomic wreath-zigzag algebra.  

We summarize the main results of the paper. We will define notation rigorously in later sections; for now, we just indicate where the definitions can be found.

Suppose that $\bla$ and $\bmu$ are both $r$-multipartitions of $n$ lying in a Rouquier block, with $\bmu$ $e$-regular, and that $\bla$ and $\bmu$ have a common multicore. For $0 \le k \le r-1$ and $0 \leq i \le e-1$, let $\la^k_i$ (resp. $\mu^k_i$) denote the quotient on runner $i$ of the abacus configuration of the $k$th component of $\bla$ (resp. $\bmu$); we define these terms in Section~\ref{SS3}. Set

\[
g_{\bla\bmu}(v) = v^{\omega(\bla)-\omega(\bmu)} \sum_{\bal \in \M{r}{e+1}} \sum_{\bbe \in \M{r}{e}} \sum_{\bgam \in \M{{r+1}}{e}} \sum_{\bdel \in \M{r}{e}} \left( \prod_{k=0}^{r-1} \prod_{i=0}^{e-1} c^{\delta^k_i}_{\mu^k_i \gamma^k_i} c^{\delta^k_i}_{ \al^k_i \be^k_i \gamma^{k+1}_i} c^{\la^k_i}_{\be^k_i (\al^k_{i+1})'}\right) \\ c^{\emp}_{\gamma^0_0 \gamma^0_1 \ldots \gamma^{0}_{e-1}}  c^{\emp}_{\gamma^r_0 \gamma^r_1 \ldots \gamma^r_{e-1}}.
\]
We define $\omega(\bla)$ in Section~\ref{SS3}. 
The terms $c^{\la}_{\al_1\ldots\al_{t}}$ are generalized Littlewood-Richardson coefficients, which are defined in Section~\ref{S:LR}, and $\Gamma^{s}_f$ is the set of $(s \times f)$-matrices whose entries are partitions; we give more details in Section~\ref{S:Coeffs}. 

Let $\{s_{\bla} \mid \bla \text{ an $r$-multipartition}\}$ denote the standard basis of the Fock space representation $\F^{\mc}$ of $\U_{v}(\widehat{\mathfrak{sl}}_{e})$, which appears in Section~\ref{SS:Fock}. For each $e$-regular $r$-multipartition $\bmu$, let $G(\bmu)$ denote the canonical basis vector of $\F^{\mc}$ indexed by $\bmu$ and suppose that
\[G(\bmu) = \sum_{\bla} d^{\mc}_{\bla\bmu}(v) s_{\bla}.\]

\begin{thmx} \label{thm:all}
Suppose that $\bla$ and $\bmu$ are $r$-multipartitions of $n$ with $\bmu$ $e$-regular. Suppose further that $\bla$ and $\bmu$ belong to a Rouquier block and that they have the same multicore. Then
\[d^{\mc}_{\bla\bmu}(v) = g_{\bla\bmu}(v).\]
\end{thmx} 

Applying (the graded version of) Ariki's Theorem~\cite{Ariki,BK:Decomp} to Theorem~\ref{thm:all}, we immediately obtain the following result.

\begin{thmx} \label{thm:all2}
  Let $\h=\h_{r,n}(q,{\bf Q})$ be an Ariki-Koike algebra defined over a field of characteristic $0$ and suppose that $\bla$ and $\bmu$ are $r$-multipartitions of $n$ with $\bmu$ a Kleshchev multipartition. Suppose further that $\bla$ and $\bmu$ belong to a Rouquier block and that they have the same multicore. Then the graded multiplicity of the simple module $D^{\bmu}$ as a composition factor of the Specht module $S^{\bla}$ is
  \[[S^{\bla}:D^{\bmu}]_v = g_{\bla\bmu}(v).\]
\end{thmx}

Using the cyclotomic $q$-Schur algebra, we may extend this result to fields of prime characteristic. 

\begin{thmx} \label{T:Charp}
Let $\h=\h_{r,n}(q,{\bf Q})$ be an Ariki-Koike algebra defined over a field of characteristic $p > 0$ and suppose that $\bla$ and $\bmu$ are $r$-multipartitions of $n$ with $\bmu$ a Kleshchev multipartition. Suppose further that $\bla$ and $\bmu$ belong to a Rouquier block, that they have the same multicore and that no component of $\bmu$ has $p$ or more removable $e$-rim hooks. 
Then the graded multiplicity of the simple module $D^{\bmu}$ as a composition factor of the Specht module $S^{\bla}$ is
  \[[S^{\bla}:D^{\bmu}]_v = g_{\bla\bmu}(v).\]
  \end{thmx}

Let $\mathcal{S}_{r,n}$ denote the cyclotomic $q$-Schur algebra corresponding to an Ariki-Koike algebra $\h_{r,n}(q,{\bf Q})$; see~\cite{Mathas:AKSurvey} for details. For $\bla$ an $r$-multipartition of $n$, let $\Delta(\bla)$ denote the Weyl module indexed by $\bla$ and $L(\bla)$ the simple module indexed by $\bla$. We prove the next conjecture in the case that $\bmu$ is $e$-regular. 

\begin{conjx} \label{C:Schur}
Suppose that $\mathcal{S}_{r,n}$ is a cyclotomic $q$-Schur algebra as above defined over a field of characteristic $p \geq 0$. Suppose that $\bla$ and $\bmu$ are $r$-multipartitions of $n$ which lie in a Rouquier block and which have the same multicore and that no component of $\bmu$ has $p$ or more removable $e$-rim hooks. Then the multiplicity of $L(\bmu)$ as a composition factor of $\Delta(\bla)$ is given by 
  \[[\Delta(\bla):L(\bmu)] = \sum_{\bal \in \M{r}{e+1}} \sum_{\bbe \in \M{r}{e}} \sum_{\bgam \in \M{{r+1}}{e}} \sum_{\bdel \in \M{r}{e}} \left( \prod_{k=0}^{r-1} \prod_{i=0}^{e-1} c^{\delta^k_i}_{\mu^k_i \gamma^k_i} c^{\delta^k_i}_{ \al^k_i \be^k_i \gamma^{k+1}_i} c^{\la^k_i}_{\be^k_i (\al^k_{i+1})'}\right) c^{\emp}_{\alpha^0_0 \al^1_0 \ldots \al^{r-1}_{0}} c^{\emp}_{\gamma^0_0 \gamma^0_1 \ldots \gamma^{0}_{e-1}}   c^{\emp}_{\gamma^r_0 \gamma^r_1 \ldots \gamma^r_{e-1}}.
\]
  \end{conjx}

We prove Theorem~\ref{thm:all} using a variation on an algorithm of Fayers~\cite{Fayers:LLT}. The variation is very similar to the method employed in~\cite{JLM} to determine the decomposition numbers for the Rouquier blocks when $r=1$ in that we find elements $f^{(s,j)} \in \U_v(\widehat{\mathfrak{sl}}_e)$ whose action on the standard basis elements $s_{\bla} \in \F^{\mc}$ where the multipartitions obtained by adding $s$ $e$-rim hooks to $[\bla]$ belong to a Rouquier block is particularly nice. Combined with results of Fayers, this enables us to compute the transition coefficients $d^{\mc}_{\bla\bmu}(v)$ under the assumptions of Theorem~\ref{thm:all}. As noted above, we may then use Ariki's theorem to prove Theorem~\ref{thm:all2}.

The paper is structured as follows. In Section~\ref{S:2}, we recall the definitions that we will need, including background on the Ariki-Koike algebras, multipartitions and abacuses, Rouquier blocks, the Fock space representation and Ariki's theorem, and Littlewood-Richardson coefficients. Although we include some new results which will be used later, most of this material is standard.

In Section~\ref{S:Decomp} we move on to the proof of the main results. Section~\ref{S:Coeffs} introduces the coefficients $g_{\bla\bmu}(v)$ and considers some of their properties. In Section~\ref{SS:FockInd}, we define the elements $f^{(s,j)}$ and in Proposition~\ref{L:I1} we describe the action of $f^{(s,j)}$ on a standard basis element $s_{\btau}$, where the resulting basis elements are indexed by multipartitions lying in a Rouquier block. This enables us to describe the action of $f^{(s,j)}$ on a truncated canonical basis element $Q(\bnu)$, which is given in Proposition~\ref{P:HeavyGraded}. Having set up our machinery, we then prove Theorem~\ref{thm:all} in Section~\ref{SS:ProofMain}. The proof is given by a triple induction, firstly on $r$, secondly on the number of removable $e$-rim hooks in the first component of $\bmu$ and lastly using a total order on the first component of $\bmu$.

In Section~\ref{SS:Schur}, we introduce the cyclotomic $q$-Schur algebra. Working in this algebra and applying the Schur functor, we prove Theorem~\ref{T:Charp}. In Section~\ref{SS:Scopes}, we note how our results also apply to blocks which are Scopes equivalent to the Rouquier blocks. 
Finally, in Section~\ref{SS:Open} we briefly discuss some open questions related to our results.

\section{Background and definitions} \label{S:2}
\subsection{Multipartitions and Young diagrams} \label{SS1}
For $n\geq 0$, we define a partition of $n$ to be a sequence $\la=(\la_1,\la_2,\ldots)$ of non-negative integers such that $\la_1 \geq \la_2 \geq \ldots$ and $\sum_{i \geq 1} \la_i =n$. When writing a partition, we usually omit the zeros at the end and gather together equal terms, so that $(3,3,2,1,1,0,0,\ldots) = (3^2,2,1^2)$. Let $\La_n$ denote the set of partitions of $n$ and let $\La=\bigcup_{n \geq 0} \La_n$ denote the set of all partitions. We write $\emp$ to denote the unique partition of $0$. If $\la \in \La_n$, we write $|\la|=n$. 

For $r \geq 1$, we say that an $r$-multipartition, or multipartition, of $n$ is an $r$-tuple of partitions $\bla=(\la^{(0)},\la^{(1)},\ldots,\la^{(r-1)})$ such that $\sum_{k=0}^{r-1} |\la^{(k)}|=n$; we write $|\bla|=n$. Note that, contrary to the usual convention, the components of our multipartitions are labelled from $0$ to $r-1$ rather than from $1$ to $r$. We write $\Larn$ to denote the set of $r$-multipartitions of $n$ and $\La^r=\bigcup_{n \geq 0} \Larn$ to denote the set of all $r$-multipartitions. Let $\emp^r$ denote the unique $r$-multipartition of $0$. 
For $e \geq 2$, we say that $\la \in \La$ is $e$-regular if no $e$ non-zero parts of $\la$ are the same, and we say that $\bla \in \La^r$ is $e$-regular if each component of $\bla$ is an $e$-regular partition. Let $\Lar$ denote the set of $e$-regular $r$-multipartitions.  

Suppose that $\bla \in \La^r$. Fix $e \geq 2$ and let $I=\{0,1,\ldots,e-1\}$; we will identify $I$ with $\Z/e\Z$.  Suppose $\mc =(a_0,a_1,\ldots,a_{r-1})\in I^r$. The Young diagram of $\bla$ is the set 
\[[\bla] = \{(x,y,k) \in \Z_{>0} \times \Z_{>0} \times \{0,1 \ldots r-1\} \mid y \leq \la_x^{(k)}\}.\]
To each node $(x,y,k) \in [\bla]$ we associate its residue, $\res_{\mc}(x,y,k) = a_k + y - x \mod e$; using the identification above we assume that $\res_{\mc}(x,y,k) \in I$. We draw the residue diagram of $\bla$ by replacing each node in the Young diagram by its residue.
We define the residue set of $\bla$ to be the multiset $\Res_{\mc}(\bla) = \{\res(\mathfrak{n}) \mid \mathfrak{n} \in [\bla]\}$. 

We say that $\mathfrak{n} \in [\bla]$ is a removable node of $[\bla]$ if the diagram containing the nodes $[\bla]\setminus \mathfrak{n}$ is the diagram of a multipartition. Similarly we say that $\mathfrak{n} \notin [\bla]$ is an addable node of $[\bla]$ if the diagram containing the nodes $[\bla] \cup \mathfrak{n}$ is the diagram of a multipartition. If a removable (resp. addable) node has residue $i$, we refer to it as a removable (resp. addable) $i$-node. 
Let $\rem_i(\bla)$ (resp. $\add_i(\bla))$ denote the set of removable (resp. addable) $i$-nodes of $\bla$. If $\mathfrak{n_1}=(x,y,k), \mathfrak{n_2}=(x',y',k') \in \Z_{>0} \times \Z_{>0} \times \{0,1,\ldots,r-1\}$ then we say that $\mathfrak{n}_1$ is above $\mathfrak{n}_2$ if $k<k'$ or if $k=k'$ and $x<x'$. 

The rim of $[\bla]$ consists of the nodes $\{(x,y,k) \in [\bla] \mid (x+1,y+1,k) \notin [\bla]\}$ and if $l \geq 1$ then an $l$-rim hook of $[\bla]$ is a connected subset of the rim of size $l$. 

\subsection{The Ariki-Koike algebra} \label{SS2}
We refer the reader to~\cite{Mathas:AKSurvey} for a survey of the Ariki-Koike algebras and to~\cite{Kleshchev:Survey} for the relationship between the Ariki-Koike algebras and the cyclotomic Khovanov-Lauda-Rouquier algebras.

Fix $r\geq 1$ and $n \geq 0$ and let $\mathbb{F}$ be a field of characteristic $p \geq 0$. Suppose $q \in \mathbb{F} \setminus \{0\}$ and ${\bf Q}=(Q_0,\dots,Q_{r-1}) \in \mathbb{F}^r$. The Ariki-Koike algebra $\h=\h_{r,n}(q,\bf Q)$ is the unital associative $\mathbb{F}$-algebra with generators $T_0 , T_1, \dots, T_{n-1}$ and relations
$$\begin{array}{crcll}
& (T_i +q )(T_i -1) & = & 0, &  \text{ for } 1 \leq i \leq n-1, \\
& T_i T_j & = & T_j T_i, & \text{ for } 0 \leq i,j \leq n-1, |i-j|>1, \\
& T_i T_{i+1} T_i & = & T_{i+1} T_i T_{i+1}, & \text{ for } 1 \leq i \leq n-2,\\
& (T_0 - Q_0)\dots (T_0 - Q_{r-1}) & = & 0, & \\
& T_0 T_1 T_0 T_1 & = & T_1 T_0 T_1 T_0. &
\end{array}$$

Define $e$ to be minimal such that $1+q+\dots+q^{e-1}=0$, or set $e=\infty$ if no such value exists. Throughout this paper we shall assume that $e$ is finite and we shall refer to $e$ as the quantum characteristic of $\h$. Write $I=\{0,1,\ldots,e-1\}$. We shall further assume that each $Q_k$ is a power of $q$, that is, there exists $\mc =(a_0,a_1,\ldots,a_{r-1}) \in I^r$ such that $Q_k = q^{a_k}$ for $0 \leq k \leq r-1$. We call $\mc$ the multicharge corresponding to $\h$.  

For each multipartition $\bla \in \Larn$ we define a $\h$-module $S^{\bla}$ called a Specht module and when $\h$ is semisimple the set $\{S^{\bla} \mid \bla \in \Larn\}$ forms a complete set of non-isomorphic irreducible $\h$-modules.
There is a subset $\Kl \subseteq \Larn$ such that the simple modules arise as the heads of the Specht modules in the set $\{S^{\bmu}\mid \bmu \in \Kl\}$. We denote the simple modules as $D^{\bmu}$ so that $\{D^{\bmu} \mid \bmu \in \Kl\}$ is a complete set of non-isomorphic irreducible $\h$-modules.

The Specht modules we use are dual to the Specht modules defined in~\cite{DJM:CellularBasis} so that $\Kl$ is a subset of the $e$-regular multipartitions; see~\cite{Fayers:WeightsII} for the connections between the two conventions. We use this definition for two reasons: firstly so that our description of the decomposition numbers is consistent with the well-known formula for decomposition numbers of Rouquier blocks when $r=1$~\cite{CT,JLM,LeclercMiyachi2} and secondly so that our notation is consistent with an algorithm of Fayers~\cite{Fayers:LLT} which informs our computations. 
We let $\Kla = \bigcup_{n\ge 0} \Kl$ and call the elements of $\Kla$ Kleshchev multipartitions, although elsewhere in the literature, our Kleshchev multipartitions are often refered to as conjugate Kleshchev multipartitions. If $r=1$, they are simply the $e$-regular partitions. Otherwise, these is a recursive method that will test whether $\bla \in \La^r$ is a Kleshchev multipartition. A Kleshchev multipartition is always $e$-regular so we have the inclusions $\Kla \subseteq \Lar \subseteq \La^r$. 

If $\bla \in \Larn$ and $\bmu \in \Kl$, let $[S^{\bla}:D^{\bmu}]$ denote the multiplicity of the simple module $D^{\bmu}$ as a composition factor of the Specht module $S^{\bla}$; these numbers are called decomposition numbers.  Determining the decomposition numbers is one of the most important open problems in the representation theory of the Ariki-Koike algebras. Even when $r=1$, there are very few cases when a closed formula for decomposition numbers is known. We note that the decomposition numbers depend only on $e,p$ and $\mc$, not on the actual values of $q$ and ${\bf Q}$. 

In~\cite{BK:Blocks}, it was shown that the Ariki-Koike algebra $\h$ is isomorphic to a cyclotomic Khovanov-Lauda-Rouquier algebra of type $A$. These are $\Z$-graded algebras and it was further shown in~\cite{BKW} that there is a corresponding $\Z$-grading on the Specht modules. Using these results one may define graded decomposition numbers $[S^{\bla}:D^{\bmu}]_v \in \mathbb{N}[v,v^{-1}]$; we recover the original decomposition numbers by setting $v=1$. For more details, see~\cite{Kleshchev:Survey}. 

All of the composition factors of a Specht module $S^{\bla}$ belong to the same block, and so we can think about a Specht module lying in a specific block or two Specht modules lying in the same block; it is clear that if $[S^{\bla}:D^{\bmu}] \neq 0$ then $S^{\bla}$ and $S^{\bmu}$ lie in the same block. 

\begin{proposition}[\cite{LM:Blocks}, Theorem~2.11] \label{P:Blocks}
Suppose that $\h=\h_{r,n}(q,{\bf Q})$ has multicharge $\mc$. 
Let $\bla,\bmu \in \Larn$. Then $S^{\bla}$ and $S^{\bmu}$ lie in the same block of $\h$ if and only if $\Res_{\mc}(\bla)=\Res_{\mc}(\bmu)$. 
\end{proposition}

\subsection{The abacus} \label{SS3}
The abacus was first introduced by James~\cite{James:Abacus} as a way to represent partitions. Fix $e \geq 2$ and let $I=\{0,1,\ldots,e-1\}$.  
Suppose that $\la=(\la_1,\la_2,\ldots) \in \La$ and $a \in I$. Define the $\beta$-set \[\BB_a(\la)=\{\la_i - i + a \mid i \geq 1\}.\]
We encode a partition as an abacus configuration. Our abacus will have $e$ runners, labelled $0,1,\ldots,e-1$ from left to right. The positions on the abacus are labelled in order by the elements of $\Z$ such that the runners congruent to $i$ modulo $e$ lie on runner $i$, for $0 \leq i \leq e-1$. Given $\la \in \La$ and $a \in \Z$, we form the resulting abacus configuration by putting a bead at each element of the $\beta$-set $\BB_a(\la)$.

\begin{ex}
Let $\la=(10^2,8,4,2,1^3)$ and let $a=3$. Then 
\[\BB_a(\la)=(12,11,8,3,0,-2,-3,-4,-6,-7,-8,\ldots).\]
If $e=5$ then the corresponding abacus configuration is given by
 \[\abacus(bbbbb,nbbbn,bnnbn,nnnbn,nbbnn,nnnnn)\]
where we assume that the runners extend infinitely up and down the page with beads above the levels shown and empty positions below the levels shown.   
\end{ex}

Given an abacus configuration corresponding to a partition $\la$ and $a \in I$, it is straightforward to recover $a$ and hence $\la$. Let $\BB$ be the $\beta$-set of the configuration and choose $x$ such that $b \in \BB$ for all $b < xe$. If $M=\#\{b \in \BB: b \ge xe\}$ then $a \equiv M \mod e$. In the example above, with $e=5$, we can count 13 beads and so, as expected, $a=3$. 

If $r=1$ it is well-known that removing a $e$-rim hook from a Young diagram corresponds to moving a bead up one position on the corresponding abacus configuration. The partition $\bar{\la}$ obtained by removing all possible $e$-rim hooks of $\la$, or equally, by moving all beads up as high as possible on the abacus configuration, is called the $e$-core (or core) of $\la$ and the number of hooks removed to get to the core is called the $e$-weight (or weight) of $\la$.

\begin{proposition}[\cite{Brauer, Robinson}] \label{P:Naka}
Suppose that $r=1$, that $\mc=(a)$ and that $\la,\mu \in \La$. Then 
$\Res_{\mc}((\la))=\Res_{\mc}((\mu))$ if and only $\la$ and $\mu$ have the same core and the same weight. 
  \end{proposition}

We now want to translate this to higher levels. Let $\mc \in I^r$ and $\bla \in \La^r$. The abacus configuration of $\bla$ with respect to $\mc$ is the $r$-tuple of abacuses where abacus $k$ has $\beta$-numbers $\BB_{a_k}(\la^{(k)})$, for $0 \leq k \leq r-1$. Erasing a removable $i$-node from component $k$ of $[\bla]$ corresponds to moving a bead on runner $i$ of abacus $k$ back by one position, that is, to runner $i-1$  (or from runner $0$ to runner $e-1$). 

In an analogue of the $r=1$ case, we define the multicore $\bar{\bla}$ of a multipartition $\bla$ to be the multipartition obtained by removing all possible $e$-rim hooks from all components of $[\bla]$. We define $\hook(\bla) = (|\bla|-|\bar{\bla}|)/e$ to be the number of $e$-rim hooks removed in order to get to the multicore.

We now define two equivalence relations on the set $\La^r$. Let $\mc \in I^r$ and suppose $\bla, \bmu \in \La^r$. 
\begin{itemize}
\item Say that $\bla \sim_{\mc} \bmu$ if $\Res_{\mc}(\bla) = \Res_{\mc}(\bmu)$.
\item Say that $\bla \approx_{\mc} \bmu$ if $\bla$ and $\bmu$ have the same multicore and $\hook(\bla)=\hook(\bmu)$.  
  \end{itemize}

In fact, the relation $\approx_{\mc}$ is independent of $\mc$ as the process of adding and removing $e$-rim hooks from a multipartition does not depend on the multicharge. Following Proposition~\ref{P:Blocks}, we refer to the $\sim_{\mc}$-equivalence classes as blocks. Proposition~\ref{P:Naka} shows that if $\bla \ap \bmu$ then $\bla \sim_{\mc} \bmu$. The converse is true (again by Proposition~\ref{P:Naka}) if $r=1$; but is false in general for $r>1$.

\begin{ex}
  Let $r=2$ and $e=3$ and suppose $\mc=(0,1)$. Let $\bla=((9,7,5,3,1^2),(2^2,1^2))$. Then the $\sim_{\mc}$-equivalence class of $\bla$ splits into five $\approx_{\mc}$-equivalence classes with representatives below. 
  
  \[\abacus(bbb,nbb,nnb,nnb,nnb,nnb) \;\;\abacus(bbb,bbb,nbb,nbb,nnn,nnn),\quad
  \abacus(bbb,nbb,nbb,nbb,nnn,nnn) \;\; \abacus(bbb,bbb,nnb,nnb,nnb,nnb), \quad
  \abacus(bbb,nbb,bbb,nnb,nnn,nnn)\;\; \abacus(bbb,nbb,nbb,nnb,nnb,nnb),\quad
  \abacus(bbb,nbb,bnb,nnb,nnb,nnn)\;\; \abacus(bbb,nbb,nbb,nbb,nnb,nnn), \quad
  \abacus(nbb,bbb,nbb,nnb,nnb,nnn)\;\; \abacus(bbb,nbb,bbb,nnb,nnb,nnn)
  \]
\end{ex} 

Suppose now that $\bla \in \La^r$ and $\mc \in I^r$. For each $0 \leq k \leq r-1$, consider the $\beta$-numbers $\BB_{a_k}(\la^{(k)})$. For $0 \leq i \leq e-1$, let $\BB^k_i=\{b \in \BB_{a_k}(\la^{(k)}) \mid b \equiv i \mod e\}$ and let $\mathfrak{C}^k_i=\{(b-i)/e \mid b \in \BB^k_i\}$. Then each $\mathfrak{C}^k_i$ is a $\beta$-set and so corresponds to a partition which we denote by $\la^k_i$. We define the quotient of $\bla$ to be the $r$-tuple of elements of $\La^e$ given by
\[\qt(\bla) = ((\la^0_0,\la^0_0,\ldots,\la^0_{e-1}),(\la^1_0,\la^1_1,\ldots,\la^1_{e-1}),\ldots,(\la^{r-1}_0,\la^{r-1}_1,\ldots,\la^{r-1}_{e-1})).\]
Note that
\[\hook(\bla) = \sum_{k=0}^{r-1} \sum_{i=0}^{e-1} |\la^k_i|.\]
We also define
\[\omega(\bla)  = \sum_{k=0}^{r-1} \sum_{i=0}^{e-1} (e-i+k-1) |\la^k_i|;\]
we will see this function again in Section~\ref{S:Decomp}.
Note that if $C$ is a $\approx_{\mc}$-equivalence class of $\La^r$ with $\hook(\bla)=h$ for any $\bla \in C$ then the elements of $C$ are in bijection with the tuples
 \[((\la^0_0,\la^0_0,\ldots,\la^0_{e-1}),(\la^1_0,\la^1_1,\ldots,\la^1_{e-1}),\ldots,(\la^{r-1}_0,\la^{r-1}_1,\ldots,\la^{r-1}_{e-1}))\]
 with $\sum_{k=0}^{r-1} \sum_{i=0}^{e-1} |\la^k_i|=h$.
 
\subsection{Rouquier blocks} \label{SS:Rouq}
The Rouquier blocks for Ariki-Koike algebras were recently introduced by the author as a generalization of the Rouquier blocks previously defined when $r=1$~\cite{L:Rouquier}. The definition we give here is equivalent. Fix quantum characteristic $e \geq 2$ and let $I=\{0,1,\ldots,e-1\}$. We fix a multicharge $\mc \in I^r$ and let $\bla \in \La^r$. For $0 \leq i \leq e-1$ and $0 \leq k \leq r-1$, define
\[\mathfrak{b}^k_i(\bla) = \max\left\{\frac{(b-i)}{e}\mid b \in \BB_{a_k}(\bar{\la}^{(k)}) \text{ and } b \equiv i \mod e\right\}.\]
We can think of $\mathfrak{b}^k_i(\bla)$ as the lowest level in component $k$ of the abacus configuration of $\bar{\bla}$ which contains a bead. Note that if $\bla\approx_{\mc} \bmu$ then $\mathfrak{b}^k_i(\bla)=\mathfrak{b}^k_i(\bmu)$ for all $i,k$. For  $1 \leq i \leq e-1$ and $0 \leq k \leq r-1$, define
\[\mathfrak{d}^k_i(\bla) = \mathfrak{b}^k_i(\bla) - \mathfrak{b}^k_{i-1}(\bla).\]
We say that $\bla$ is a Rouquier multipartition if 
\[\hook(\bla) \leq \mathfrak{d}^k_i(\bla) +1\]
for all $1 \leq i \leq e-1$ and $0 \leq k \leq r-1$. We say that a $\sim_{\mc}$-equivalence class $B$ of $\La^r$ is a Rouquier block if each $\bla \in B$ is a Rouquier multipartition. As observed in~\cite{L:Rouquier}, if $r \geq 2$, it is perfectly possible to have $\bla \sim_{\mc} \bmu$ with $\bla$ a Rouquier multipartition and $\bmu$ not. But if $\bla \approx_{\mc} \bmu$ then $\bla$ is a Rouquier multipartition if and only if $\bmu$ is a Rouquier multipartition. 

Let $\R^{\mc} \subset \La^r$ denote the set of multipartitions that belong to Rouquier blocks.  We return to the last example. 

\begin{ex}
  Let $r=2$ and $e=3$ and suppose $\mc=(0,1)$. Let $\bla=((9,7,5,3,1^2),(2^2,1^2))$. Then the $\sim_{\mc}$-equivalence class of $\bla$ is a Rouquier block. To see this, it is sufficient to consider the five $\approx_{\mc}$-equivalence classes. 
  
  \[\abacus(bbb,nbb,nnb,nnb,nnb,nnb) \;\;\abacus(bbb,bbb,nbb,nbb,nnn,nnn),\quad
  \abacus(bbb,nbb,nbb,nbb,nnn,nnn) \;\; \abacus(bbb,bbb,nnb,nnb,nnb,nnb), \quad
  \abacus(bbb,nbb,bbb,nnb,nnn,nnn)\;\; \abacus(bbb,nbb,nbb,nnb,nnb,nnb),\quad
  \abacus(bbb,nbb,bnb,nnb,nnb,nnn)\;\; \abacus(bbb,nbb,nbb,nbb,nnb,nnn), \quad
  \abacus(nbb,bbb,nbb,nnb,nnb,nnn)\;\; \abacus(bbb,nbb,bbb,nnb,nnb,nnn)
  \]
  For each multipartition $\bmu$ in the first class (which only contains one multipartition) we have 
\[\mathfrak{d}^0_1(\bmu)=1, \qquad \mathfrak{d}^0_2(\bmu)=4, \qquad \mathfrak{d}^1_0(\bmu)=2, \qquad \mathfrak{d}^1_1(\bmu)=0, \qquad \hook(\bmu) =0.\] 
Similarly we can see that the other four classes contain only Rouquier multipartitions.  
\end{ex} 

In fact, it follows from~\cite[Corollary~5.2]{DA} that it is sufficient to check the Rouquier condition on the classes obtained by adding hooks to the `core block' of the equivalence class, that is, the classes with the maximal number of removable $e$-rim hooks. 

\begin{lemma} [\cite{CT2}, Lemma~4.1 (1)] \label{L:Rouqreg}
  Suppose that $\bmu\in \R$ with
   \[\qt(\bmu) = ((\mu^0_0,\mu^0_1,\ldots,\mu^0_{e-1}),(\mu^1_0,\mu^1_1,\ldots,\mu^1_{e-1}),\ldots,(\mu^{r-1}_0,\mu^{r-1}_1,\ldots,\mu^{r-1}_{e-1})).\]
Then $\bmu$ is $e$-regular if and only if $\mu^k_0=\emp$ for $0 \leq k\leq r-1$.   
\end{lemma}

\subsection{The Fock space representation of $\mathcal{U}_v(\widehat{\mathfrak{sl}}_e)$} \label{SS:Fock}
Let $e \geq 2$, let $I=\{0,1,\ldots,e-1\}$ and let $\mathcal{U}$ denote the quantized enveloping algebra $\mathcal{U}=\mathcal{U}_v(\widehat{\mathfrak{sl}}_e)$. Given $\mc \in I^r$, there exists a representation $\mathcal{F}^{\mc}$ of $\mathcal{U}$ whose basis $\{s_{\bla} \mid \bla \in \La^r\}$ is indexed by multipartitions and where the action of $\mathcal{U}$ depends on $\mc$. We call $\mathcal{F}^{\mc}$ the Fock space representation of $\mathcal{U}$. The $\U$-submodule $M^{\mc}$ generated by $s_{\emp^r}$ is isomorphic to the irreducible highest weight module $V(\Upsilon)$ for some dominant integral weight $\Upsilon$ of $\mathcal{U}$. This module has a canonical basis (in the sense of Lusztig and Kashiwara) and so we may write the canonical basis elements in terms of the standard basis above. 
The canonical basis elements can be indexed by the set of Kleshchev multipartitions; following the conjecture by Lascoux, Leclerc and Thibon~\cite{LLT}, Ariki~\cite{Ariki} proved that the transition coefficients between the canonical basis and the standard basis evaluated at $v=1$ are exactly the decomposition numbers $[S^{\bla}:D^{\bmu}]$ which appear in the representation theory of the Ariki-Koike algebras, where the algebras are defined over a field of characteristic $0$. It was later shown~\cite{BK:Decomp} that in fact the transition coefficients are equal to the graded decomposition numbers $[S^{\bla}:D^{\bmu}]_v$. 

We are therefore interested in computing the canonical basis vectors for $M^{\mc}$. There are various algorithms which will perform this computation; we use the main ideas of an algorithm of Fayers~\cite{Fayers:LLT} which computes canonical basis vectors for a $\mathcal{U}$-module $M^{\otimes \mc}$ where $M^{\mc} \subseteq M^{\otimes \mc} \subseteq \mathcal{F}^{\mc}$. 

Below we describe only the concepts that we need for this paper; for a full description of the algebra $\U$ and its action on the Fock space $\F^{\mc}$, we refer the reader to~\cite{LLT}. They may also find it useful to refer to~\cite{Fayers:LLT}. 

Let $e\geq 2$ and set $I=\{0,1,\ldots, e-1\}$, which we identify with $\Z/e\Z$. Let $\mathcal{U} = \mathcal{U}_v(\widehat{\mathfrak{sl}}_e)$. This is a $\mathbb{Q}(v)$-algebra with generators $e_i, f_i$ for $i \in I$ and $v^h$ for $h \in P^{\vee}$; the relations may be found in~\cite{LLT}.  There exists a $\mathbb{Q}$-linear automorphism $^{-}:\U \rightarrow \U$ called the bar involution determined by
\[\overline{e_i} = e_i, \qquad \overline{f_i} = f_i, \qquad \overline{v}=v^{-1}, \qquad \overline{v^h} = v^{-h}.\]
Let $\mc \in I^r$ and let $\F^{\mc}$ be the $\mathbb{Q}(v)$-vector space with basis $\{s_{\bla} \mid \bla \in \Lambda^r\}$. This becomes a $\U$-module under the action described in~\cite{LLT}; it is sufficient for us to describe the action of $f_i$ which we do below. 
Let $M^{\mc}$ be the submodule generated by $s_{\emp^r}$. We can define a bar involution on $M^{\mc}$ which is compatible with the bar involution on $\U$ by setting 
\[\overline{s_{\emp^r}} = s_{\emp^r} \qquad \text{ and }\qquad \overline{um} = \overline{u} \, \overline{m} \text{ for all } u \in \U \text{ and } m \in M^{\mc}.\] 
Then $M^{\mc}$ has a canonical basis $\{G^{\mc}(\bmu) \mid \bmu \in \Kla\}$ which is uniquely determined subject to the following properties. For $\bmu \in \Kla$, 
\begin{itemize}
\item $\overline{G^{\mc}(\bmu)} = G^{\mc}(\bmu)$.
\item Suppose $G^{\mc}(\bmu) = \sum_{\bla \in \La^r} d^{\mc}_{\bla\bmu}(v)s_{\bla}$ where $d^{\mc}_{\bla\bmu}(v) \in \mathbb{Q}(v)$. Then
\begin{itemize}
\item $d^{\mc}_{\bmu\bmu}(v)=1$; and
\item $d^{\mc}_{\bla\bmu}(v) \in v \Z[v]$ for $\bla \neq \bmu$.  
\end{itemize}
\end{itemize}

We can extend the bar involution on $M^{\mc}$ to the whole of $\F^{\mc}$, and so we have a canonical basis $\{G^{\mc}(\bla) \mid \bla \in \La^r\}$ for $\F^{\mc}$ which satisfies the conditions above. 
For $\bla,\bmu \in \La^r$, we define $d^{\mc}_{\bla\bmu}(v)\in \mathbb{Q}(v)$ to be the coefficient which appears in the sum
\[G^{\mc}(\mu) = \sum_{\bla \in \La^r} d^{\mc}_{\bla\bmu}(v) s_{\bla}.\]
Ariki's theorem relates the coefficients $d^{\mc}_{\bla\bmu}(v)$ evaluated at $v=1$ to the decomposition numbers for the Ariki-Koike algebras over fields of characteristic $0$. Brundan and Kleshchev later showed that $d^{\mc}_{\bla\bmu}(v)$ is exactly the graded decomposition number. 

\begin{theorem} [\cite{Ariki,BK:Decomp}] \label{T:Ariki}
Let $\h=\h_{r,n}(q,{\bf Q})$ be an Ariki-Koike algebra over a field of characteristic $0$, where $\h$ has quantum characteristic $e$ and multicharge $\mc$. 
Suppose that $\bla,\bmu \in \Larn$ with $\bmu \in \Kl$. Then
\[[S^{\bla}:D^{\bmu}]_v = d^{\mc}_{\bla\bmu}(v).\]
\end{theorem}

Let $\F^\circ = \{x \in \F^{\mc} \mid \overline{x} = x\}$ denote the set of bar-invariant elements of $\F^{\mc}$. Then \[\F^\circ = \left\{\sum_{\bmu \in \La^r} b_{\bmu}(v) G^{\mc}(\bmu) \mid b_{\bmu}(v) \in \mathbb{Q}(v) \text{ and } \overline{b_{\bmu}(v)}=b_{\bmu}(v)\right\}.\]

\begin{corollary} \label{C:Ones}
Suppose that $x = \sum_{\bla \in \La^r} c_{\bla} s_{\bla} \in \F^\circ$ where $c_{\bla} \in \Z[v]$. Then
\[x = \sum_{\substack{\bmu \in \La^r \\ c_{\bmu} \in \Z}} c_{\bmu} G^{\mc}(\bmu).\]
\end{corollary}

As shown in~\cite[Proposition~4.2]{Fayers:LLT}, the space $M^{\otimes \mc}$ with basis 
$\{G^{\mc}(\bmu) \mid \bmu \in \Lar\}$ 
is a $\U$-submodule of $\F^{\mc}$. Since every Kleshchev multipartition is $e$-regular, we have $M^{\mc} \subset M^{\otimes \mc} \subset \F^{\mc}$.  
Fayers described an algorithm to compute the coefficients $d^{\mc}_{\bla\bmu}(v)$ where $\bmu \in \Lar$. The following results are easy consequences of his algorithm. 

\begin{lemma} \label{L:DomNeeded}
Suppose that $\bla,\bmu \in \La^r$ with $\bmu \in \Lar$.
\begin{itemize}
\item $d^{\mc}_{\bla\bmu}(v) = 0$ unless $\Res_{\mc}(\bla)=\Res_{\mc}(\bmu)$. In particular, $\bmu$ and $\bla$ must have the same size. 
\item Suppose that $|\bmu|=|\bla|$. Then $d^{\mc}_{\bmu\bmu}=1$ and $d^{\mc}_{\bla\bmu}(v) = 0$ unless $|\mu^{(0)}| \geq |\la^{(0)}|$. 
\end{itemize}
\end{lemma} 

Now suppose $r \ge 2$. If $\bla = (\la^{(0)},\la^{(1)},\ldots,\la^{(r-1)}) \in \La^r$, let $\hat{\bla}=(\la^{(1)},\ldots,\la^{(r-1)}) \in \La^{r-1}$. If $\mc=(a_0,a_1,\ldots,a_{r-1}) \in I^r$, write $\hat{\mc}= (a_1,\ldots,a_{r-1})$. 

\begin{lemma} [\cite{Fayers:LLT}, Corollary~3.2] \label{L:AddEmpty}
Suppose $r \geq 2$ and that $\bla\in \La^r$, $\bmu \in \Lar$ with $\mu^{(0)}=\la^{(0)}=\emp$. Then
\[d^{\mc}_{\bla\bmu}(v) = d^{\hat{\mc}}_{\hat{\bla}\hat{\bmu}}(v).\]
\end{lemma}

For $i \in I$ and $m >0$, we write $f^{(m)}_i = f^{m}_i / [m]!$, the quantum divided power.
We describe the action of the divided powers $f_i^{(m)}$ on the standard basis of $\F^{\mc}$. 
If $\bnu \in \La^r$, write $\bnu \xrightarrow{m:i} \bla$ if $[\bla]$ is formed from $[\bnu]$ by adding $m$ nodes all of residue $i$. If $\bnu \xrightarrow{m:i} \bla$ set
\[N(\bnu,\bla) =\sum_{\mathfrak{n} \in [\bla] \setminus [\bnu]} \#\{\mathfrak{m} \in \add_i(\bnu) \mid \mathfrak{m} \text{ is above } \mathfrak{n}\}-\#\{\mathfrak{m} \in \rem_i(\bla) \mid \mathfrak{m} \text{ is above } \mathfrak{n}\}. \]
Then
\[f_i^{(m)} s_{\bnu} = \sum_{\bnu \xrightarrow{m:i} \bla} v^{N(\bnu,\bla)} s_{\bla}.\] 

As we have seen previously, adding a node of residue $i$ to a multipartition $\bnu$ corresponds to moving a bead on the abacus configuration of $\bnu$ from runner $i-1$ to runner $i$. By identifying the basis elements $s_{\bnu}$ with abacus configurations, we observe the action of $f^{(m)}_i$ on $s_{\bnu}$.

\begin{ex} Let $r=2$ and $e=3$. Take $\mc=(2,1)$ and $\bnu=((4,3^5,2,1),(2^4,1^3))$. Consider $f^{(2)}_1 s_{\bnu}$. 

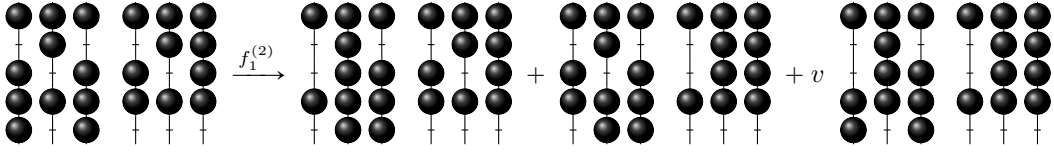
\[\abacus(bbb,nbn,bnb,bbb,bnb) \;\;\;\abacus(bbb,nbb,bnb,bbb,nnn) \; \xrightarrow{f^{(2)}_1} \;
\abacus(bbb,nbn,nbb,bbb,nbb) \;\;\;\abacus(bbb,nbb,bnb,bbb,nnn)
\; + \; \abacus(bbb,nbn,bnb,bbb,nbb) \;\;\;\abacus(bbb,nbb,nbb,bbb,nnn)
\; +v \;\;\abacus(bbb,nbn,nbb,bbb,bnb) \;\;\;\abacus(bbb,nbb,nbb,bbb,nnn)
\]
\end{ex}

Note that it is straightforward to read $N(\bnu,\bla)$ from the abacus. Let $\bnu \in \La^r$ and for $0 \leq k \leq r-1$ let $\BB^{(k)}=\BB_{a_k}(\nu^{(k)})$ be the set of $\beta$-numbers for $\nu^{(k)}$. 
Then there is a bijection between addable $i$-nodes of $\bnu$ and elements of the set
\[\{(b,k) \in \Z \times \{0,1,\ldots,r-1\} \mid b \notin \BB^{(k)}, \; b+1 \in \BB^{(k)} \text{ and } b \equiv i-1 \mod e\}\]
with $(b,k)$ lying above $(b',k')$ if and only if $k<k'$ or $k=k'$ and $b<b'$. A similar construction exists for removable $i$-nodes.  

The next example illustrates Lemma~\ref{L:I1}. 

\begin{ex}
 Let $r=2$ and $e=3$. Take $\mc=(0,2)$ and $\bnu=((6,4,2^2,1^2),(7,5,3^2,2^2,1^2))$. Consider $f^{(2)}_1 f^{(2)}_2 f^{(2)}_0 s_{\bnu}$. 

\begin{multline*}
\abacus(bbb,bbb,nbb,nbb,nnb,nnb,nnn) \;\;\;\abacus(bbb,bbb,nbb,nbb,nbb,nnb,nnb) \; \xrightarrow{f^{(2)}_1 f^{(2)}_2 f^{(2)}_0} \; 
\abacus(bbb,bbb,nnb,nbb,nbb,nnb,nnn) \;\;\;\abacus(bbb,bbb,nbb,nbb,nbb,nnb,nnb)
\; + v \;\;\abacus(bbb,nbb,bbb,nnb,nbb,nnb,nnn) \;\;\;\abacus(bbb,bbb,nbb,nbb,nbb,nnb,nnb)
\; + v^2 \;\; \abacus(bbb,nbb,nbb,bbb,nnb,nnb,nnn) \;\;\;\abacus(bbb,bbb,nbb,nbb,nbb,nnb,nnb)\\
+ v \;\; \abacus(bbb,bbb,nbb,nnb,nbb,nnb,nnn) \;\;\;\abacus(bbb,bbb,nbb,nbb,nnb,nbb,nnb)
\; + v^2 \;\; \abacus(bbb,bbb,nbb,nnb,nbb,nnb,nnn) \;\;\;\abacus(bbb,nbb,bbb,nbb,nbb,nnb,nnb) 
\; +v^2 \;\; \abacus(bbb,nbb,bbb,nbb,nnb,nnb,nnn) \;\;\;\abacus(bbb,bbb,nbb,nbb,nnb,nbb,nnb)
\; +v^3\;\;\abacus(bbb,nbb,bbb,nbb,nnb,nnb,nnn) \;\;\;\abacus(bbb,nbb,bbb,nbb,nbb,nnb,nnb) \\
+ v^2 \;\; \abacus(bbb,bbb,nbb,nbb,nnb,nnb,nnn) \;\;\;\abacus(bbb,bbb,nbb,nnb,nbb,nbb,nnb)
\; + v^3 \;\; \abacus(bbb,bbb,nbb,nbb,nnb,nnb,nnn) \;\;\;\abacus(bbb,nbb,bbb,nbb,nnb,nbb,nnb)
\; + v^4 \;\; \abacus(bbb,bbb,nbb,nbb,nnb,nnb,nnn) \;\;\;\abacus(bbb,nbb,nbb,bbb,nbb,nnb,nnb)
\end{multline*}
\end{ex}

The final result, which follows from~\cite{LLT}, is used in Fayers' algorithm~\cite{Fayers:LLT}.   

\begin{lemma}  [\cite{LLT}] \label{L:LLTLead}
Suppose that $\bmu \in \La^r$ with $\mu^{(0)}$ $e$-regular. Let $\hat{\bmu}^{\emp} = (\emp,\mu^{(1)},\ldots,\mu^{(r-1)})$. Then there exists $f = f^{(t_x)}_{i_x} f^{(t_{x-1})}_{i_{x-1}} \ldots f^{(t_1)}_{i_{1}} \in \U$ such that 
 \[f s_{\hat{\bmu}^\emp} = s_{\bmu} + 
\sum_{\substack{\bla = (\la,\mu^{(1)},\ldots,\mu^{(r-1)}) \\ \la \sim_{(a_0)} \mu^{(0)}, \, \la \neq \mu^{(0)}}} b_{\bla}(v) s_{\bla}
+ \sum_{\substack{\bla \sim \bmu \\ |\la^{(0)}| < |\mu^{(0)}|}} b_{\bla}(v) s_{\bla}\]
where $b_{\bla}(v) \in \mathbb{N}[v,v^{-1}]$ for all $\bla$. 
 In particular, if $\mu^{(0)}$ is an $e$-core then the middle term above is empty. 
\end{lemma}

\subsection{Littlewood-Richardson coefficients} \label{S:LR}
The Littlewood-Richardson coefficients appear naturally in the representation theory of the symmetric groups $\mathfrak{S}_n$ over the complex numbers. 
We give a brief introduction to them here; for more details, and for a combinatorial description of how to compute them, we refer the reader to~\cite[Chapter~5]{Fulton}. 

For $\la \in \La_n$ we let $S^{\la}_{\C}$ denote the irreducible $\C \mathfrak{S}_n$-module corresponding to $\la$. If $\al,\be \in \La$ with $|\al|=n_1$ and $|\be|=n_2$ then we have
\[S^{\al}_{\C} \otimes S^{\be}_{\C} \uparrow_{\mathbb{C}(\mathfrak{S}_{n_1} \times \mathfrak{S}_{n_2})}^{\mathbb{C} \mathfrak{S}_{n_1 + n_2}} = \sum_{\la \vdash n_1+n_2} c^{\la}_{\al\be}S^{\la}_{\C} \]
for some non-negative integers $c^{\la}_{\al\be}$. If $\al,\be,\la \in \La$ with $|\al|+|\be|=|\la|$ then $c^{\la}_{\al\be}$ is the corresponding Littlewood-Richardson coefficient. We extend the definition to arbitrary $\al,\be,\la \in \La$ by setting $c^{\la}_{\al\be}=0$ if $|\al|+|\be|\neq|\la|$.

Since the tensor product is associative, we can generalize this definition. 
Suppose that $t \geq 1$ and $\alpha_1,\alpha_2,\ldots,\alpha_t \in \Lambda$ with $\alpha_i \vdash n_i$ for each $i$. Then 
\[S^{\alpha_1}_{\C} \otimes S^{\alpha_2}_{\C} \otimes \ldots \otimes S^{\alpha_t}_{\C} \uparrow_{\mathbb{C}(\mathfrak{S}_{n_1} \times \mathfrak{S}_{n_2} \times \ldots \times \mathfrak{S}_{n_t})}^{\mathbb{C} \mathfrak{S}_{n_1+n_2+\ldots+n_t}} = \sum_{\lambda \vdash n_1+\ldots+n_t} c^{\lambda}_{\alpha_1\alpha_2\ldots\alpha_t} S^{\la}_{\C}\]
for some non-negative integers $c^{\lambda}_{\alpha_1\alpha_2\ldots\alpha_t}$, and we call $c^{\lambda}_{\alpha_1\alpha_2\ldots\alpha_t}$ a generalized Littlewood-Richardson coefficient. If $|\lambda| \neq |\alpha_1|+\ldots+|\alpha_t|$ we set $c^{\la}_{\alpha_1\alpha_2\ldots\alpha_t}=0$.

Using the properties of the tensor product, we can see that the generalized Littlewood-Richardson coefficients satisfy the following recursive equation. 
We have $c^{\la}_{\al} = \delta_{\al \la}$ and for $t \geq 2$ we have 
\[c^{\la}_{\alpha_1\alpha_2\ldots\alpha_t} = \sum_{\beta\in \La} c^\la_{\al_1 \beta} c^\beta_{\alpha_2\ldots\al_t};\]
note that for $t=2$, this formula is self-referential. Also note that if $\la,\al \in \La$ then $c^{\la}_{\al \emp} = \delta_{\la \al}$.  

For the rest of this section, we use the convention that unless otherwise stated, all sums are over $\Lambda$, with the understanding that only partitions of the right size contribute to the sum. 
Recall that if $\la$ is a partition of $n$ then $\la'$ is the conjugate partition, that is, the partition obtained by swapping the rows and the columns of $[\la]$. 

\begin{lemma} [\cite{James}, Theorem~6.7] \label{L:Conjugation}
Suppose $\la \in \La$. Then 
$$S_{\C}^{\la'} \cong S_{\C}^{\la} \otimes S_{\C}^{(1^n)},$$ 
where $S^{(1^n)}_{\C}$ is the 1-dimensional sign representation. 
\end{lemma}

\begin{lemma} \label{LR1}
Suppose that $\al,\be,\gamma,\la \in \La$. Then
\begin{align*}
c^{\la'}_{\al'\be'} & = c^{\la}_{\al\be}=c^{\la}_{\be\al},  \\ 
\sum_{\sigma}c^{\la}_{\al\sigma} c^{\sigma}_{\be\gamma} & = \sum_{\sigma} c^{\la}_{\be \sigma} c^{\sigma}_{\al\gamma} = \sum_{\sigma} c^{\la}_{\gamma \sigma} c^{\sigma}_{\al \be}. \end{align*}
\end{lemma}

\begin{proof}
The equalities in the first equation follow from Lemma~\ref{L:Conjugation} and the symmetry of the tensor product. Then all the terms in the second equation are equal to $c^{\la}_{\al\be\gamma}$. 
\end{proof}

The proof of the following result  again uses the properties of the tensor product. 

\begin{lemma} [\cite{CT}, Lemma~2.2] \label{LR2}
Let $a,b,\sigma \in \La$ and let $t \geq 0$. Then 
\[\sum_{\mu} c^{\mu}_{ab}c^\mu_{\sigma(1^t)} = \sum_{l=0}^t \sum_{\alpha,\beta} c^{\sigma}_{\al' \be} c^{a'}_{\alpha(l)} c^{b}_{\be(1^{t-l})}. 
\]
\end{lemma}

The next result follows immediately from the definition of the Littlewood-Richardson coefficients.

\begin{lemma} \label{L:OneS1}
  Let $\nu \in \La_n$ and $s \geq 1$. Suppose $\la \in \La_{n+s}$. Then 
\begin{itemize}
\item $c^{\la}_{\nu(1^s)} = 1$ if $[\nu] \subset [\la]$ and no two nodes of $[\la] \setminus [\nu]$ lie in the same row, and $c^{\la}_{\nu(1^s)}=0$ otherwise.
\item $c^{\la}_{\nu(s)} = 1$ if $[\nu] \subset [\la]$ and no two nodes of $[\la] \setminus [\nu]$ lie in the same column, and $c^{\la}_{\nu(s)}=0$ otherwise.
\end{itemize}
\end{lemma}

We can rephrase Lemma~\ref{L:OneS1} in terms of $\beta$-numbers.

\begin{lemma} \label{L:OneS}
  Let $\nu \in \La_n$ and $s \geq 1$. Let $a \in I$ and suppose that $\BB=\BB_{a}(\nu)$. Suppose that $\la \in \La_{n+s}$ and that $\mathfrak{C}=\BB_{a}(\la)$. Suppose that
  \[\{b \in \BB \mid b \notin \mathfrak{C}\} = \{b_1,b_2,\ldots,b_t\}, \qquad \qquad \{c \in \mathfrak{C} \mid c \notin \BB\} = \{c_1,c_2,\ldots,c_t\},\]
  where $b_{t'}<b_{t'+1}$ and $c_{t'} < c_{t'+1}$ for $1 \leq t' <t$;  
  note that these sets do have the same size and that they satisfy $\sum_{t'=1}^t (c_{t'}-b_{t'})=s$. Then
  \begin{itemize}
  \item $c^{\la}_{\nu(1^s)}=1$ if
    \[b_1 < c_1<b_2<c_2 < \ldots < b_t < c_t \text{ and if } b_{t'}<b<c_{t'} \text{ for some $t'$ then } b \in \BB \cap \mathfrak{C};\]
    and $c^{\la}_{\nu(1^s)}=0$ otherwise.
  \item $c^{\la}_{\nu(s)}=1$ if
    \[b_1 < c_1<b_2<c_2 < \ldots < b_t < c_t \text{ and if } b_{t'}<b<c_{t'} \text{ for some $t'$ then } b \notin \BB \cup \mathfrak{C};\]
    and $c^{\la}_{\nu(s)}=0$ otherwise.   
    \end{itemize} 
\end{lemma}

We end with a more complicated result that we will use later. If $\bd \in \La^t$ we use the convention that $\bd = (d^0,d^1,\ldots,d^{t-1})$. 

\begin{lemma}\label{L:Tunnel}
Let $s \geq 0$ and $r \geq 1$ and suppose that $0 \leq m \leq r-1$. Suppose $\bet, \bnu \in \La^r$. Then
\begin{multline*}
\sum_{\bgam \in \La^{r+1}} \sum_{\bdel \in \La^r} \sum_{{\bep} \in \La^r} \sum_{z_0+\ldots+z_{r-1}=s} 
\left( \prod_{k=0}^{r-1} c^{\delta^k}_{\nu^k \gamma^k}c^{\delta^k}_{\gamma^{k+1}\ep^k} c^{\eta^k}_{\ep^k (1^{z_k})} \right) c^{\emp}_{\gamma^r} \\
= \sum_{\bgam \in \La^{r}} \sum_{\bdel \in \La^r} \sum_{{\bep} \in \La^m}\sum_{\substack{z_0+\ldots+z_{m-1}\\ +Z=s}} 
\left( \prod_{k=0}^{m-1} c^{\delta^k}_{\nu^k \gamma^k}c^{\delta^k}_{\gamma^{k+1}\ep^k} c^{\eta^k}_{\ep^k (1^{z_k})} \right) 
c^{\delta^{m}}_{\nu^m \gamma^m(1^Z)} \left(\prod_{k=m+1}^{r-1} c^{\delta^{k-1}}_{\gamma^{k}\eta^{k-1}} c^{\delta^k}_{\nu^k \gamma^k}\right) c^{\eta^{r-1}}_{\delta^{r-1}}. 
\end{multline*} 
In particular, setting $m=0$ we have 
\begin{multline*}
\sum_{\bgam \in \La^{r+1}} \sum_{\bdel \in \La^r} \sum_{{\bep} \in \La^r} \sum_{z_0+\ldots+z_{r-1}=s} 
\left( \prod_{k=0}^{r-1} c^{\delta^k}_{\nu^k \gamma^k}c^{\delta^k}_{\gamma^{k+1}\ep^k} c^{\eta^k}_{\ep^k (1^{z_k})} \right) c^{\emp}_{\gamma^r} 
= \sum_{\bgam \in \La^{r}} \sum_{\bdel \in \La^r} c^{\delta^0}_{\nu^0 \gamma^0 (1^s)} 
\left(\prod_{k=1}^{r-1} c^{\delta^{k-1}}_{\gamma^{k}\eta^{k-1}} c^{\delta^k}_{\nu^k \gamma^k}\right) c^{\eta^{r-1}}_{\delta^{r-1}}.
\end{multline*}
\end{lemma}

\begin{proof}
We have
\begin{align*}
\sum_{\bgam \in \La^{r+1}} \sum_{\bdel \in \La^r} \sum_{{\bep} \in \La^r}& \sum_{z_0+\ldots+z_{r-1}=s} 
\left( \prod_{k=0}^{r-1} c^{\delta^k}_{\nu^k \gamma^k}c^{\delta^k}_{\gamma^{k+1}\ep^k} c^{\eta^k}_{\ep^k (1^{z_k})} \right) c^{\emp}_{\gamma^r} \\
& = \sum_{\bgam \in \La^{r+1}} \sum_{\bdel \in \La^r} \sum_{{\bep} \in \La^r} \sum_{z_0+\ldots+z_{r-1}=s} 
\left( \prod_{k=0}^{r-2} c^{\delta^k}_{\nu^k \gamma^k}c^{\delta^k}_{\gamma^{k+1}\ep^k} c^{\eta^k}_{\ep^k (1^{z_k})} \right) 
c^{\delta^{r-1}}_{\nu^{r-1}\gamma^{r-1}} c^{\delta^{r-1}}_{\gamma^r \ep^{r-1}} c^{\eta^{r-1}}_{\ep^{r-1}(1^{z_{r-1}})}
c^{\emp}_{\gamma^r} \\
& = \sum_{\bgam \in \La^{r}} \sum_{\bdel \in \La^r} \sum_{{\bep} \in \La^{r-1}} \sum_{z_0+\ldots+z_{r-1}=s} 
\left( \prod_{k=0}^{r-2} c^{\delta^k}_{\nu^k \gamma^k}c^{\delta^k}_{\gamma^{k+1}\ep^k} c^{\eta^k}_{\ep^k (1^{z_k})} \right) 
c^{\delta^{r-1}}_{\nu^{r-1}\gamma^{r-1}} c^{\eta^{r-1}}_{\delta^{r-1}(1^{z_{r-1}})} \\
& = \sum_{\bgam \in \La^{r}} \sum_{\bdel \in \La^{r-1}} \sum_{{\bep} \in \La^{r-1}} \sum_{\substack{z_0+\ldots+z_{r-2}\\+ Z =s}} 
\left( \prod_{k=0}^{r-2} c^{\delta^k}_{\nu^k \gamma^k}c^{\delta^k}_{\gamma^{k+1}\ep^k} c^{\eta^k}_{\ep^k (1^{z_k})} \right) 
c^{\eta^{r-1}}_{\nu^{r-1}\gamma^{r-1}(1^Z)}\\
& = \sum_{\bgam \in \La^{r}} \sum_{\bdel \in \La^r} \sum_{{\bep} \in \La^{r-1}} \sum_{\substack{z_0+\ldots+z_{r-2}\\+ Z =s}} 
\left( \prod_{k=0}^{r-2} c^{\delta^k}_{\nu^k \gamma^k}c^{\delta^k}_{\gamma^{k+1}\ep^k} c^{\eta^k}_{\ep^k (1^{z_k})} \right) 
c^{\delta^{r-1}}_{\nu^{r-1}\gamma^{r-1}(1^Z)} c^{\eta^{r-1}}_{\delta^{r-1}}\\
\end{align*}
and so the result holds for $m=r-1$. Suppose now that $0 \leq m \leq r-2$ and that the equation holds for $m+1$. Then
\begin{align*}
\sum_{\bgam \in \La^{r+1}}& \sum_{\bdel \in \La^r} \sum_{{\bep} \in \La^r}  \sum_{z_0+\ldots+z_{r-1}=s} 
\left( \prod_{k=0}^{r-1} c^{\delta^k}_{\nu^k \gamma^k}c^{\delta^k}_{\gamma^{k+1}\ep^k} c^{\eta^k}_{\ep^k (1^{z_k})} \right) c^{\emp}_{\gamma^r} \\
& = \sum_{\bgam \in \La^{r}} \sum_{\bdel \in \La^r} \sum_{{\bep} \in \La^{m+1}}\sum_{\substack{z_0+\ldots+z_{m}\\ +Z=s}} 
\left( \prod_{k=0}^{m} c^{\delta^k}_{\nu^k \gamma^k}c^{\delta^k}_{\gamma^{k+1}\ep^k} c^{\eta^k}_{\ep^k (1^{z_k})} \right) 
c^{\delta^{m+1}}_{\nu^{m+1} \gamma^{m+1}(1^Z)} \left(\prod_{k=m+2}^{r-1} c^{\delta^{k-1}}_{\gamma^{k}\eta^{k-1}} c^{\delta^k}_{\nu^k \gamma^k}\right) c^{\eta^{r-1}}_{\delta^{r-1}} \\
& = \sum_{\bgam \in \La^{r}} \sum_{\bdel \in \La^r} \sum_{{\bep} \in \La^{m+1}}\sum_{\substack{z_0+\ldots+z_{m-1}\\ +Z_a+Z_b=s}} 
\left( \prod_{k=0}^{m-1} c^{\delta^k}_{\nu^k \gamma^k}c^{\delta^k}_{\gamma^{k+1}\ep^k} c^{\eta^k}_{\ep^k (1^{z_k})} \right)  \\ 
& \qquad \qquad c^{\delta^m}_{\nu^m \gamma^m}c^{\delta^m}_{\gamma^{m+1}\ep^m} c^{\eta^m}_{\ep^m (1^{Z_a})}
c^{\delta^{m+1}}_{\nu^{m+1} \gamma^{m+1}(1^{Z_b})} \left(\prod_{k=m+2}^{r-1} c^{\delta^{k-1}}_{\gamma^{k}\eta^{k-1}} c^{\delta^k}_{\nu^k \gamma^k}\right) c^{\eta^{r-1}}_{\delta^{r-1}} \\
& = \sum_{\bgam \in \La^{r}} \sum_{\bdel \in \La^r} \sum_{{\bep} \in \La^{m+1}}\sum_{\substack{z_0+\ldots+z_{m-1}\\ +Z_a+Z_b=s}} 
\left( \prod_{k=0}^{m-1} c^{\delta^k}_{\nu^k \gamma^k}c^{\delta^k}_{\gamma^{k+1}\ep^k} c^{\eta^k}_{\ep^k (1^{z_k})} \right)  \\ 
& \qquad \qquad \sum_g c^{\delta^m}_{\nu^m \gamma^m}c^{\delta^m}_{\gamma^{m+1}\ep^m} c^{\eta^m}_{\ep^m (1^{Z_a})}
c^{\delta^{m+1}}_{\nu^{m+1} g} c^g_{\gamma^{m+1}(1^{Z_b})} 
\left(\prod_{k=m+2}^{r-1} c^{\delta^{k-1}}_{\gamma^{k}\eta^{k-1}} c^{\delta^k}_{\nu^k \gamma^k}\right) c^{\eta^{r-1}}_{\delta^{r-1}} \\
& =  \sum_{\gamma^0,\ldots,\gamma^{m}} \sum_{\gamma^{m+2},\ldots,\gamma^{r-1}} \sum_{\bdel \in \La^r} \sum_{{\bep} \in \La^{m}}\sum_{\substack{z_0+\ldots+z_{m-1}\\ +Z_a+Z_b=s}} 
\left( \prod_{k=0}^{m-1} c^{\delta^k}_{\nu^k \gamma^k}c^{\delta^k}_{\gamma^{k+1}\ep^k} c^{\eta^k}_{\ep^k (1^{z_k})} \right)  \\ 
& \qquad \qquad \sum_g c^{\delta^m}_{\nu^m \gamma^m}  c^{\delta^{m+1}}_{\nu^{m+1} g}  \left( \sum_{\gamma^{m+1}, \epsilon^m} c^{\delta^m}_{\gamma^{m+1}\ep^m} c^{\eta^m}_{\ep^m (1^{Z_a})}
c^g_{\gamma^{m+1}(1^{Z_b})} \right)
\left(\prod_{k=m+2}^{r-1} c^{\delta^{k-1}}_{\gamma^{k}\eta^{k-1}} c^{\delta^k}_{\nu^k \gamma^k}\right) c^{\eta^{r-1}}_{\delta^{r-1}} \\
& = \sum_{\gamma^0,\ldots,\gamma^{m}} \sum_{\gamma^{m+2},\ldots,\gamma^{r-1}}\sum_{\bdel \in \La^r} \sum_{{\bep} \in \La^{m}}\sum_{\substack{z_0+\ldots+z_{m-1}\\ +Z=s}} 
\left( \prod_{k=0}^{m-1} c^{\delta^k}_{\nu^k \gamma^k}c^{\delta^k}_{\gamma^{k+1}\ep^k} c^{\eta^k}_{\ep^k (1^{z_k})} \right)  \\ 
& \qquad \qquad \sum_g c^{\delta^m}_{\nu^m \gamma^m}
c^{\delta^{m+1}}_{\nu^{m+1} g} \sum_d c^d_{g \eta^m} c^d_{\delta^m(1^Z)} 
\left(\prod_{k=m+2}^{r-1} c^{\delta^{k-1}}_{\gamma^{k}\eta^{k-1}} c^{\delta^k}_{\nu^k \gamma^k}\right) c^{\eta^{r-1}}_{\delta^{r-1}} \qquad \text{ by Lemma~\ref{LR2}} \\
& = \sum_{\gamma^0,\ldots,\gamma^{m}} \sum_{\gamma^{m+2},\ldots,\gamma^{r-1}}\sum_{\delta^0,\ldots,\delta^{m-1}} \sum_{\delta^{m+1},\ldots,\delta^{r-1}} \sum_{{\bep} \in \La^{m}}\sum_{\substack{z_0+\ldots+z_{m-1}\\ +Z=s}} 
\left( \prod_{k=0}^{m-1} c^{\delta^k}_{\nu^k \gamma^k}c^{\delta^k}_{\gamma^{k+1}\ep^k} c^{\eta^k}_{\ep^k (1^{z_k})} \right)  \\ 
& \qquad \qquad \sum_{g,d} c^{d}_{\nu^m \gamma^m(1^Z)}
c^d_{g \eta^m}  c^{\delta^{m+1}}_{\nu^{m+1} g} 
\left(\prod_{k=m+2}^{r-1} c^{\delta^{k-1}}_{\gamma^{k}\eta^{k-1}} c^{\delta^k}_{\nu^k \gamma^k}\right) c^{\eta^{r-1}}_{\delta^{r-1}} \\
& = \sum_{\bgam \in \La^{r}} \sum_{\bdel \in \La^r} \sum_{{\bep} \in \La^m}\sum_{\substack{z_0+\ldots+z_{m-1}\\ +Z=s}} 
\left( \prod_{k=0}^{m-1} c^{\delta^k}_{\nu^k \gamma^k}c^{\delta^k}_{\gamma^{k+1}\ep^k} c^{\eta^k}_{\ep^k (1^{z_k})} \right) 
c^{\delta^{m}}_{\nu^m \gamma^m(1^Z)} \left(\prod_{k=m+1}^{r-1} c^{\delta^{k-1}}_{\gamma^{k}\eta^{k-1}} c^{\delta^k}_{\nu^k \gamma^k}\right) c^{\eta^{r-1}}_{\delta^{r-1}}, 
\end{align*}
and so the equation holds for $m-1$. Hence by induction the equation holds for all $0 \leq m \leq r-1$, and setting $m=0$ we get the final result. 
\end{proof}

\section{Decomposition numbers} \label{S:Decomp}
Throughout this section, we fix the parameters $r$ and $e$. 
If $\mc \in I^r$ is a multicharge, recall that $\R^{\mc}$ is the set of multipartitions which belong to Rouquier blocks (with respect to that multicharge). If $\bla, \bmu \in \R^{\mc}$ then $\bla \approx_{\mc} \bmu$ if $\bla$ and $\bmu$ have the same size and the same multicore.

From an $r$-tuple of abacus configurations we can read off not just the corresponding multipartition but also the underlying multicharge. Since we think of the Rouquier multipartitions in terms of abacus configurations, a Rouquier block carries with it an inherent multicharge. Hence from now on, we set $\R=\bigcup_{\mc \in I^r} \R^{\mc}$ and when we talk about a Rouquier block we shall assume that the multicharge is understood. If $\bla, \bmu \in \R^{\mc}$ for some $\mc \in I^r$, we will drop the $\mc$ in our notation and write $\bla \sim \bmu$ instead of $\bla \sim_{\mc} \bmu$, $d_{\bla\bmu}(v)$ instead of $d^{\mc}_{\bla\bmu}(v)$ and so forth. Set $\RR=\R\cap \Lar$ to be the set of $e$-regular multipartions which belong to Rouquier blocks. 
 Set
\[\Rp=\{(\bla,\bmu) \in \R \times \RR \mid \bla \approx \bmu\}.\]
In this section we determine a closed formula for $d_{\bla\bmu}(v)$, where $(\bla,\bmu) \in \Rp$, in terms of sums of products of Littlewood-Richardson coefficients. By Theorem~\ref{T:Ariki}, this formula also describes the graded decomposition number $[S^{\bla}:D^{\bmu}]_v$ where $\h$ is defined over a field of characteristic $0$.  
We begin by defining coefficients $g_{\bla\bmu}(v)$ and proving some properties about them; we  then show that they are equal to the transition coefficients $d_{\bla\bmu}(v)$. 

\subsection{The coefficients $g_{\bla\bmu}(v)$} \label{S:Coeffs}
Throughout this section, we fix $r \geq 1$ and $e \geq 2$. If $\bla \in \R$, we will assume that 
\[\qt(\bla)=((\la^0_0,\la^0_1,\ldots,\la^0_{e-1}),(\la^1_{0},\la^1_1,\ldots,\la^1_{e-1}),\ldots,(\la^{r-1}_0,\la^{r-1}_1,\ldots,\la^{r-1}_{e-1}))\]
and we shall use $\la^k_i$ to denote the $i$th part of the quotient corresponding to $\la^{(k)}$ without further explanation. When the multicore is understood, we shall abuse notation by identifying $\bla$ with its quotient. 

For $s,f \geq 1$, let $\M{s}{f}$ denote the set of $s\times f$ matrices whose entries are partitions. The rows of each matrix are indexed by the set $\{0,1,\ldots,s-1\}$ and the columns by the set $\{0,1,\ldots,f-1\}$. For $\bal \in \Gamma^s_f$ and $ 0 \leq k \leq s-1, 0 \leq i \leq f-1$, we let $\alpha^k_i \in \Lambda$ denote the entry in the row indexed by $k$ and the column indexed by $i$. 

Suppose that $(\bla,\bmu) \in \Rp$. Recall from Lemma~\ref{L:Rouqreg} that this implies that $\mu^k_0=\emp$ for all $0 \leq k \leq r-1$. Define

\begin{multline}\label{defn:f}
g_{\bla\bmu}(v) = v^{\omega(\bla)-\omega(\bmu)} \sum_{\bal \in \M{r}{e+1}} \sum_{\bbe \in \M{r}{e}} \sum_{\bgam \in \M{{r+1}}{e}} \sum_{\bdel \in \M{r}{e}} \left( \prod_{k=0}^{r-1} \prod_{i=0}^{e-1} c^{\delta^k_i}_{\mu^k_i \gamma^k_i} c^{\delta^k_i}_{ \al^k_i \be^k_i \gamma^{k+1}_i} c^{\la^k_i}_{\be^k_i (\al^k_{i+1})'}\right)  c^{\emp}_{\gamma^0_0 \gamma^0_1 \ldots \gamma^{0}_{e-1}}  c^{\emp}_{\gamma^r_0 \gamma^r_1 \ldots \gamma^r_{e-1}}
\end{multline}

We find it helpful to arrange the terms in the double product into an array. 

\begin{ex} Let $r=2$ and $e=3$ and suppose $(\bla,\bmu) \in \Rp$. Then 
  \[g_{\bla\bmu}(v) = v^{\omega(\bla)-\omega(\bmu)}\sum_{\bal \in \M{2}{4}} \sum_{\bbe \in \M{2}{3}} \sum_{\bgam \in \M{{3}}{3}} \sum_{\bdel \in \M{2}{3}}\left(
  \begin{array}{lllllll} 
& c^{\delta^{0}_{0}}_{\mu^{0}_{0}\gamma^{0}_{0}} &&c^{\delta^{0}_{1}}_{\mu^{0}_{1}\gamma^{0}_{1}}&& c^{\delta^{0}_{2}}_{\mu^{0}_{2}\gamma^{0}_{2}} & \\ [10pt]
&c^{\delta^{0}_{0}}_{\al^{0}_{0}\be^{0}_{0}\gamma^{1}_{0}} & c^{\la^{0}_{0}}_{\be^{0}_{0}(\al^{0}_{1})'} &
c^{\delta^{0}_{1}}_{\al^{0}_{1}\be^{0}_{1}\gamma^{1}_{1}} & c^{\la^{0}_{1}}_{\be^{0}_{1}(\al^{0}_{2})'} &
c^{\delta^{0}_{2}}_{\al^{0}_{2}\be^{0}_{2}\gamma^{1}_{2}} & c^{\la^{0}_{2}}_{\be^{0}_{2}(\al^{0}_{3})'} \\ [10pt]
 &c^{\delta^{1}_{0}}_{\mu^{1}_{0}\gamma^{1}_{0}}&&c^{\delta^{1}_{1}}_{\mu^{1}_{1}\gamma^{1}_{1}}&& c^{\delta^{1}_{2}}_{\mu^{1}_{2}\gamma^{1}_{2}} & \\ [10pt]
&c^{\delta^{1}_{0}}_{\al^{1}_{0}\be^{1}_{0}\gamma^{2}_{0}} & c^{\la^{1}_{0}}_{\be^{1}_{0}(\al^{1}_{1})'} &
c^{\delta^{1}_{1}}_{\al^{1}_{1}\be^{1}_{1}\gamma^{2}_{1}} & c^{\la^{1}_{1}}_{\be^{1}_{1}(\al^{1}_{2})'} &
c^{\delta^{1}_{2}}_{\al^{1}_{2}\be^{1}_{2}\gamma^{2}_{2}} & c^{\la^{1}_{2}}_{\be^{1}_{2}(\al^{1}_{3})'}
 \end{array}\right)\]
where $\gamma^{0}_0 = \gamma^0_1 = \gamma^0_2 = \gamma^{2}_{0}=\gamma^{2}_{1}=\gamma^{2}_{2}=\emp$.
\end{ex}

\begin{lemma} \label{L:r1}
If $r=1$ then 
\[g_{\bla\bmu}(v) = v^{\omega(\bla)-\omega(\bmu)} \sum_{\al_0,\ldots,\al_{e}} \sum_{\be_0,\ldots,\be_{e-1}}
\prod_{i=0}^{e-1} c^{\mu^0_i}_{\al_i \be_i} c^{\la^0_i}_{\be_i (\al_{i+1})'}\]
where
\[|\al_i| = \sum_{j=0}^{i-1} \Big(|\la^{0}_j| - |\mu^0_j|\Big) \; \text{ for } 0 \leq i \leq e, \qquad |\be_i|= |\la^0_i| + \sum_{j=0}^{i} \Big(|\mu^0_j|-|\la^0_j|\Big) \;\text{ for } 0 \leq i \leq e-1,\]
and
\[\omega(\bla)-\omega(\bmu) = \sum_{i=0}^{e-1} (e-i-1) \Big( |\la^0_i| - |\mu^0_i| \Big) = \sum_{i=0}^{e-1} i \Big( |\mu^0_i| - |\la^0_i| \Big).\]  
\end{lemma}

\begin{proof}
If $r=1$ then the only terms $\gamma^{k}_i$ which contribute to the sum are those where $\gamma^k_i = \emp$ for $k=0, 1$ and $0 \leq i \leq e-1$. Then the only terms $\delta^0_i$ which contribute must satisfy $\delta^0_i=\mu^0_i$ for $0 \leq i \leq e-1$ and we see that $g_{\bla\bmu}(v)$ is indeed given by the formula indicated. The terms $\al_i$ and $\be_i$ which contribute to the sum must satisfy
\[|\al_i| + |\be_i| = |\mu^0_i|, \qquad \qquad |\be_i| + |\al_{i+1}| = |\la^0_i|, \qquad \text{ for } 0 \leq i \leq e-1\]
and so noting that $|\mu^0_0|=0$, it is clear that such $|\alpha_i|$ and $|\beta_i|$ are as stated. The first part of the formula for $\omega(\bla)-\omega(\bmu)$ follows from the definition and the second part follows because since $\bla \approx \bmu$ we have $\sum_{i=0}^{e-1} |\la^0_i|= \sum_{i=0}^{e-1}|\mu^0_i|$. 
\end{proof}

It is well-known that the formula in Lemma~\ref{L:r1} gives the graded decomposition numbers for the Rouquier blocks in characteristic $0$ when $r=1$~\cite{LeclercMiyachi2}, and this will later form the base step in our inductive argument to prove the main theorem. 

The terms $c^{\emp}_{\gamma^0_0 \gamma^0_1 \ldots \gamma^0_{e-1}}$ and
  $c^{\emp}_{\gamma^r_0 \gamma^r_1 \ldots \gamma^r_{e-1}}$ which appear at the Equation~\ref{defn:f} ensure that the only terms $\gamma^{k}_i$ which contribute to the sum satisfy $\gamma^k_i=\emp$ for $k=0, r$ and $0 \leq i \leq e-1$. This means that we do not have to specify the size of the partitions $\al^k_i$ and so forth that appear in the sum as only partitions of the right size will contribute. 

\begin{lemma} \label{l:lims}
The only terms $\bal,\bbe,\bgam,\bdel$ which contribute to Equation~\ref{defn:f} satisfy 
\begin{align*}
\sum_{i=0}^{e-1} |\gamma^k_i| & = \sum_{l=0}^{k-1} \sum_{i=0}^{e-1}\Big( |\mu^{l}_i|-|\la^{l}_i|\Big), && \text{for all } 0 \leq k \leq r-1, \\
|\al^k_i| & = \sum_{j=0}^{i-1} \Big( |\la^k_j| - |\mu^k_j| - |\gamma^k_j| + |\gamma^{k+1}_j|\Big), && \text{for all } 0 \leq k \leq r-1, \, 0 \leq i \leq e, \\
|\be^k_i| & = |\la^k_i| + \sum_{j=0}^{i} \Big(|\mu^k_j| - |\la^k_j| + |\gamma^k_j| - |\gamma^{k+1}_j|\Big), && \text{for all } 0 \leq k \leq r-1, \, 0 \leq i \leq e-1,\\ 
|\delta^k_i| & = |\mu^k_i| + |\gamma^k_i|, && \text{for all } 0 \leq k \leq r-1, \, 0 \leq i \leq e-1.
\end{align*}
\end{lemma}
  
\begin{proof} 
Assume that we have terms $\bal$, $\bbe$, $\bgam$ and $\bdel$ which contribute to Equation~\ref{defn:f}. Then for each $0 \leq k \leq r-1$ and $0 \leq i \leq e-1$ we have 
\begin{equation*} |\la^k_i| = |\be^k_i| + |\al^k_{i+1}|, \qquad \qquad |\delta^k_i| = |\mu^{k}_i| + |\gamma^k_i| = |\al^k_i| + |\be^k_i| + |\gamma^{k+1}_i|.\end{equation*}
Now noting that $\mu^k_0=\emp$ for all $0 \leq k \leq r-1$ and $\gamma^0_0 = \emp$, the second condition above implies that $\gamma^k_0=\emp$ for $0 \leq k \leq r$ and $\al^k_0=\be^k_0 =\emp$ for $0 \leq k \leq r-1$. From here, it is straightforward to see that we can indeed write each term $|\al^k_i|, |\be^k_i|, |\delta^k_i|$ as in Lemma~\ref{l:lims}. 
Now 
\begin{align*}
\sum_{k=0}^{r-1} \sum_{i=0}^{e-1} |\mu^k_i| = \sum_{k=0}^{r-1} \sum_{i=0}^{e-1} |\la^k_i| & \implies
\sum_{k=0}^{r-1} \sum_{i=0}^{e-1} \Big(|\be^k_i| + |\al^k_{i+1}| \Big) = \sum_{k=0}^{r-1} \sum_{i=0}^{e-1} \Big(|\al^k_i| + |\be^k_i| + |\gamma^{k+1}_i| - |\gamma^k_i|\Big) \\
& \implies \sum_{k=0}^{r-1} |\al^k_e| = \sum_{k=0}^{r-1} \sum_{i=0}^{e-1}\Big( |\gamma^{k+1}_i| - |\gamma^k_i|\Big) = \sum_{i=0}^{e-1} \Big(|\gamma^{e}_i| - |\gamma^0_i|\Big) =0
\end{align*}
so $\al^{k}_e=\emp$ for all $0 \leq k \leq r-1$. We now use induction on $k$ to show that the first equality of Lemma~\ref{l:lims} holds. If $k=0$, both sides of the equation are equal to $0$, so suppose that $0<k\leq r-1$ and that the equation holds for $k-1$. Then
\begin{align*}
\sum_{l=0}^{k-1} \sum_{i=0}^{e-1}\Big( |\mu^{l}_i|-|\la^{l}_i|\Big) 
& = \sum_{l=0}^{k-2} \sum_{i=0}^{e-1}\Big( |\mu^{l}_i|-|\la^{l}_i|\Big) + \sum_{i=0}^{e-1} \Big( |\mu^{k-1}_i| - |\la^{k-1}_i| \Big) \\
& = \sum_{i=0}^{e-1} |\gamma^{k-1}_i| + \sum_{i=0}^{e-1} \Big(|\alpha^{k-1}_i| + |\beta^{k-1}_i| + |\gamma^{k}_i| - |\gamma^{k-1}_{i}| - |\beta^{k-1}_{i}| - |\alpha^{k-1}_{i+1}|\Big) \\
&=  \sum_{i=0}^{e-1} |\gamma^{k}_i|
\end{align*}
and so by induction, the first equation holds for all $0 \leq k \leq r-1$. 
\end{proof}

We could give more restrictions on the size of the partitions $\gamma^k_i$ that contribute to Equation~\ref{defn:f} using the fact that $|\al^k_i| \geq 0$ and $|\be^k_i| \geq 0$ for all possible $i,k$. However, we can easily see that for $r \geq 2$, it may be that partitions of more than one size can contribute. 

\begin{ex} Let $r=2$ and $e=3$. Suppose that 
\[\qt(\bmu)=((\emp,(1),(1)),(\emp,\emp,\emp)), \qquad \qquad \qt(\bla)=((\emp,(1),\emp),(\emp,(1),\emp)).\] Then there are two values of $(\bal,\bbe,\bgam,\bdel)$ which contribute to $g_{\bla\bmu}(v)$, namely 
\begin{align*}
\bal & = \begin{pmatrix} \emp &\emp &\emp &\emp  \\ \emp &\emp &(1) &\emp \end{pmatrix}, & 
\bbe& = \begin{pmatrix} \emp & (1)& \emp \\ \emp&\emp&\emp \end{pmatrix}, &
\bgam & = \begin{pmatrix} \emp&\emp&\emp \\\emp&\emp&(1) \\ \emp&\emp&\emp \end{pmatrix}, &
\bdel & = \begin{pmatrix} \emp&(1)&(1) \\ \emp&\emp&(1) \\ \end{pmatrix} & & \text{and} \\ 
\bal & = \begin{pmatrix} \emp&\emp&(1)&\emp \\ \emp&\emp&\emp&\emp \end{pmatrix}, & 
\bbe& = \begin{pmatrix} \emp & \emp& \emp \\ \emp&(1)&\emp \end{pmatrix}, &
\bgam & = \begin{pmatrix} \emp&\emp&\emp \\\emp&(1)&\emp \\ \emp&\emp&\emp \end{pmatrix}, &
\bdel & = \begin{pmatrix} \emp&(1)&(1) \\ \emp&(1)&\emp \\ \end{pmatrix}, 
\end{align*}
and after computing the Littlewood-Richardson coefficients, we see that $g_{\bla\bmu}(v) = 2v^2$.
\end{ex}

We note that the coefficients $g_{\bla\bmu}(v)$ do satisfy the properties of the transition coefficients $d_{\bla\bmu}(v)$. It is not, we admit, actually necessary to prove this since we will later show that they really are transition coefficients; but the proof is not difficult and may satisfy the reader who prefers not to make assumptions ahead of time. 

\begin{lemma}
Suppose $(\bla,\bmu) \in \Rp$. Then $g_{\bla\bmu}(v)=1$ if $\bla=\bmu$ and $g_{\bla\bmu} \in v \mathbb{N}[v]$ otherwise.
\end{lemma}

\begin{proof}
If $\bla=\bmu$ then by Lemma~\ref{l:lims}, the only terms $\bal,\bbe,\bgam,\bdel$ which contribute to Equation~\ref{defn:f} satisfy $|\gamma^k_i|=|\alpha^k_i|=0$, $|\beta^k_i|=|\la^k_i|$ and $|\delta^k_i|=|\mu^k_i|$ for all possible values of $k,i$, so in fact must satisfy $\beta^k_i=\delta^k_i=\la^k_i=\mu^k_i$ and $\alpha^k_i=\gamma^k_i=\emp$. It is then clear that $g_{\bla\bmu}(v)=1$. So we want to show that if $\bla \neq \bmu$ and $g_{\bla\bmu}(v)\neq 0$ then $\omega(\bla) > \omega(\bmu)$. 

Take $(\bla,\bmu) \in \Rp$ and suppose that $g_{\bla\bmu}\neq 0$. Choose $\bal,\bbe,\bgam,\bdel$ as in Equation~\ref{defn:f} such that 
\[\left( \prod_{k=0}^{r-1} \prod_{i=0}^{e-1} c^{\delta^k_i}_{\mu^k_i \gamma^k_i} c^{\delta^k_i}_{ \al^k_i \be^k_i \gamma^{k+1}_i} c^{\la^k_i}_{\be^k_i (\al^k_{i+1})'}\right) \\ c^{\emp}_{\gamma^0_0 \gamma^0_1 \ldots \gamma^{0}_{e-1}}  c^{\emp}_{\gamma^r_0 \gamma^r_1 \ldots \gamma^r_{e-1}}>0,\]
so that $\bal,\bbe,\bgam,\bdel$ satisfy the conditions of Lemma~\ref{l:lims}. 
Then
\begin{align*}
\sum_{k=0}^{r-1} \sum_{i=0}^{e}|\alpha^k_i| = \sum_{k=0}^{r-1} \sum_{i=0}^{e-1}|\alpha^k_i|
& = \sum_{k=0}^{r-1} \sum_{i=0}^{e} \sum_{j=0}^{i-1} \Big(|\la^k_j| - |\mu^k_j|-|\gamma^k_j|+|\gamma^{k+1}_j| \Big) \\
&=\sum_{k=0}^{r-1} \sum_{i=0}^{e-1} (e-i-1) \Big(|\la^k_i| - |\mu^k_i|-|\gamma^k_i|+|\gamma^{k+1}_i| \Big) \\ 
&=\sum_{k=0}^{r-1} \sum_{i=0}^{e-1} (e-i-1) \Big(|\la^k_i| - |\mu^k_i| \Big)
\end{align*}
so that 
\begin{align*}
\omega(\bla) - \omega(\bmu) & = \sum_{k=0}^{r-1} \sum_{i=0}^{e-1} (e-i+k-1) \Big(|\la^k_i| - |\mu^k_i|\Big)\\ 
& = \sum_{k=0}^{r-1} \sum_{i=0}^{e}|\alpha^k_i| +  \sum_{k=0}^{r-1} \sum_{i=0}^{e-1} k \Big(|\la^k_i| - |\mu^k_i|\Big)\\ 
& = \sum_{k=0}^{r-1} \sum_{i=0}^{e}|\alpha^k_i| +  (r-1) \sum_{k=0}^{r-1} \sum_{i=0}^{e-1}  \Big(|\la^k_i| - |\mu^k_i|\Big) + \sum_{k=0}^{r-1} \sum_{i=0}^{e-1} (r-1-k) \Big(|\mu^k_i| - |\la^k_i|\Big)\\ 
&= \sum_{k=0}^{r-1} \sum_{i=0}^{e}|\alpha^k_i| + \sum_{k=0}^{r-1} \sum_{l=0}^{l} \sum_{i=0}^{e-1}\Big(|\mu^l_j| - |\la^l_j|\Big)\\
&= \sum_{k=0}^{r-1} \sum_{i=0}^{e}|\alpha^k_i| + \sum_{k=0}^{r-1} \sum_{i=0}^{e-1} |\gamma^k_i|
\end{align*}
Hence $\omega(\bla)-\omega(\bmu)\geq 0$ with equality if and only if $|\alpha^k_i|=|\gamma^k_i|=0$ for all $i,k$. But in this case, the Littlewood-Richardson coefficients are non-zero if and only if $\mu^k_i = \delta^k_i = \beta^k_i = \la^k_i$ for all $0 \leq i \leq e-1$ and $0 \leq k \le r-1$. Thus if $\bla \neq \bmu$ and $g_{\bla\bmu} \neq 0$, we do indeed have $\omega(\bla)-\omega(\bmu)>0$. 
\end{proof}

\begin{lemma} \label{L:MinusOK}
Suppose $(\bla,\bmu) \in \Rp$  where $\hook(\la^{(0)})=\hook(\mu^{(0)})=0$. Let $\hat{\bla}=(\la^{(1)},\ldots,\la^{(r-1)})$ and $\hat{\bmu}=(\mu^{(1)},\ldots,\mu^{(r-1)})$. Then
\[g_{\bla\bmu}(v) = g_{\hat{\bla}\hat{\bmu}}(v).\]
\end{lemma}

\begin{proof}
Our assumptions are that
\begin{align*}
\qt(\bla) & = ((\emp,\emp,\ldots,\emp),(\la^1_0,\la^1_1,\ldots,\la^1_{e-1}),\ldots,(\la^{r-1}_0,\la^{r-1}_1,\ldots,\la^{r-1}_{e-1})),\\
\qt(\hat{\bla}) & = ((\la^1_0,\la^1_1,\ldots,\la^1_{e-1}),\ldots,(\la^{r-1}_0,\la^{r-1}_1,\ldots,\la^{r-1}_{e-1})), \\
 \qt(\bmu) & = ((\emp,\emp,\ldots,\emp),(\emp,\mu^1_1,\ldots,\mu^1_{e-1}),\ldots,(\emp,\mu^{r-1}_1,\ldots,\mu^{r-1}_{e-1})),\\ 
\qt(\hat{\bmu}) & = ((\emp,\mu^1_1,\ldots,\mu^1_{e-1}),\ldots,(\emp,\mu^{r-1}_1,\ldots,\mu^{r-1}_{e-1})).
\end{align*}

For $s,f \geq 1$ and $\bep \in \Gamma^s_f$, define $\check{\bep} \in \Gamma^{s+1}_f$ by setting $\check{\ep}^0_i=0$ for $0 \leq i \leq f-1$ and $\check{\ep}^k_i=\ep^{k-1}_i$ for $1 \leq k \leq s$ and $0 \leq i \leq f-1$. Then
\begin{align*}
 g_{\hat{\bla}\hat{\bmu}}(1) &=  \sum_{\bal \in \M{r-1}{e+1}} \sum_{\bbe \in \M{r-1}{e}} \sum_{\bgam \in \M{{r}}{e}} \sum_{\bdel \in \M{r-1}{e}} \left( \prod_{k=0}^{r-2} \prod_{i=0}^{e-1} c^{\delta^k_i}_{\hat{\mu}^{k}_i \gamma^k_i} c^{\delta^k_i}_{ \al^k_i \be^k_i \gamma^{k+1}_i} c^{\hat{\la}^{k}_i}_{\be^k_i (\al^k_{i+1})'}\right) c^{\emp}_{\gamma^0_0 \gamma^0_1 \ldots \gamma^{0}_{e-1}}  c^{\emp}_{\gamma^r_0 \gamma^r_1 \ldots \gamma^r_{e-1}} \\
&  =  \sum_{\bal \in \M{r-1}{e+1}} \sum_{\bbe \in \M{r-1}{e}} \sum_{\bgam \in \M{{r}}{e}} \sum_{\bdel \in \M{r-1}{e}} \left( \prod_{k=0}^{r-1} \prod_{i=0}^{e-1} c^{\check{\delta}^k_i}_{\mu^k_i \check{\gamma}^k_i} c^{\check{\delta}^k_i}_{ \check{\al}^k_i \check{\be}^k_i \check{\gamma}^{k+1}_i} c^{\la^k_i}_{\check{\be}^k_i (\check{\al}^k_{i+1})'}\right) c^{\emp}_{\check{\gamma}^1_0 \check{\gamma}^1_1 \ldots \check{\gamma}^{1}_{e-1}}  c^{\emp}_{\check{\gamma}^r_0 \check{\gamma}^r_1 \ldots \check{\gamma}^r_{e-1}} \\
& =\sum_{\bal \in \M{r}{e+1}} \sum_{\bbe \in \M{r}{e}} \sum_{\bgam \in \M{{r+1}}{e}} \sum_{\bdel \in \M{r}{e}} \left( \prod_{k=0}^{r-1} \prod_{i=0}^{e-1} c^{\delta^k_i}_{\mu^k_i \gamma^k_i} c^{\delta^k_i}_{ \al^k_i \be^k_i \gamma^{k+1}_i} c^{\la^k_i}_{\be^k_i (\al^k_{i+1})'}\right) c^{\emp}_{\gamma^0_0 \gamma^0_1 \ldots \gamma^{0}_{e-1}}  c^{\emp}_{\gamma^r_0 \gamma^r_1 \ldots \gamma^r_{e-1}} \\
& =g_{\bla\bmu}(1)
\end{align*}   
where we applied Lemma~\ref{l:lims} in the penultimate step. 
 Also
 \begin{align*}
 \omega(\bla)-\omega(\bmu) & =  \sum_{k=0}^{r-1} \sum_{i=0}^{e-1} (e-i+k-1) \Big(|\la^k_i| - |\mu^k_i| \Big)\\\
 & = \sum_{k=1}^{r-1} \sum_{i=0}^{e-1} (e-i+k-1) \Big(|\la^k_i| - |\mu^k_i| \Big)\\
 &= \sum_{k=0}^{r-2} \sum_{i=0}^{e-1} (e-i+k) \Big(|\la^{k+1}_{i}| - |\mu^{k+1}_{i}| \Big)\\
&= \sum_{k=0}^{r-2} \sum_{i=0}^{e-1} (e-i+k-1) \Big(|\la^{k+1}_{i}| - |\mu^{k+1}_{i}| \Big)
+ \sum_{k=0}^{r-2} \sum_{i=0}^{e-1} \Big(|\la^{k+1}_{i}| - |\mu^{k+1}_{i}| \Big)\\
&= \sum_{k=0}^{r-2} \sum_{i=0}^{e-1} (e-i+k-1) \Big(|\la^{k+1}_{i}| - |\mu^{k+1}_{i}| \Big) \\
 &= \omega(\hat{\bla})-\omega(\hat{\bmu}) 
 \end{align*}
 and so $g_{\bla\bmu}(v) = g_{\hat{\bla}\hat{\bmu}}(v)$ as required.
\end{proof}

\subsection{Induction in the Fock space} \label{SS:FockInd}

Suppose that $1 \leq j \leq e-1$ and $s \geq 1$. Define 
\[f^{(s,j)} = f^{(s)}_{j} \ldots  f^{(s)}_{2} f^{(s)}_1  f^{(s)}_{j+1}\ldots f^{(s)}_{e-1} f^{(s)}_0 \in \U.\]
For $s \geq 1$, define $\Rs\subset\R$ to be the set of multipartitions $\bla$ lying in a Rouquier block which have the property that if we add $s$ $e$-rim hooks to $\bla$, the subsequent multipartitions still lie in $\R$; note that $\Rs$ is a union of Rouquier blocks. Set $\Rrs=\Rs \cap \RR$. 

\begin{proposition} \label{L:I1}
Suppose $s >0$ and that $\btau \in \Rs$. Let $1 \leq j \leq e-1$ and set 
\[I^j(\btau) = \{\bla \in \R \mid \bla \text{ and } \btau \text{ have the same multicore and for } 0 \leq k \leq r-1, \la^{k}_i = \tau^{k}_i \text{ unless } i = j-1, j\}.\]
Then
\[f^{(s,j)} s_{\btau} = \sum_{\bla \in I^j(\btau)} \sum_{x=0}^s \sum_{\substack{l_0+\ldots+l_{r-1}=x\\ \ell_0+\ldots + \ell_{r-1} = s-x}}\prod_{k=0}^{r-1} v^{(k+1) l_k + k \ell_k} c^{\la^{k}_{j-1}}_{\tau^{k}_{j-1} (l_k)} c^{\la^{k}_{j}}_{\tau^{k}_{j}(1^{\ell_k})}s_{\bla}.\]
\end{proposition}

\begin{proof}
When $r=1$, the ungraded version of this result is~\cite[Lemma~3.6]{JLM}. We first prove the ungraded version of Proposition~\ref{L:I1} by reducing it to the $r=1$ case. We want to consider
\[f^{(s)}_{j} \ldots  f^{(s)}_{2} f^{(s)}_1  f^{(s)}_{j+1}\ldots f^{(s)}_{e-1} f^{(s)}_0 s_{\btau}\] where $\btau$ satisfies the conditions of the proposition. Suppose that $s_{\bnu}$ occurs as one of the terms in $f^{(s)}_0 s_{\btau}$, that is, $[\bnu]$ is formed by adding $s$ nodes of residue $0$ to $[\btau]$. Suppose that $p_k$ nodes were added to $[\tau^{(k)}]$ so that $p_0+p_1+\ldots+p_{r-1}=s$. Then, by the assumptions on $\btau$, for $0 \leq k \leq r-1$, $[\nu^{(k)}]$ has at most $p_k$ addable $(e-1)$-nodes. Hence any multipartition formed from $\bnu$ by adding $s$ $(e-1)$-nodes must again add $p_k$ nodes to $[\nu^{(k)}]$ for all $0 \leq k \leq r-1$. Continuing in this way, we see that the only terms that occur in $f^{(s,j)}s_{\btau}$ are those formed by choosing some $p_0+p_1+\ldots+p_{r-1}=s$ and, for $0 \leq k\leq r-1$, essentially applying $f^{(p_k,j)}$ to the $k$th component of $\btau$. Hence~\cite[Lemma~3.6]{JLM} gives the ungraded version of Proposition~\ref{L:I1}.

Before giving the graded result, we consider an example which illustrates how these induction sequences occur. The reader may also find it helpful to read the proof of~\cite[Lemma~3.6]{JLM}. 

\begin{ex} Let $r=1$. Take $j=2$ and $s=4$. To save space, we have omitted the top lines of the abacuses. Below we have $\qt(\btau)=((\emp,\emp,(1),\emp))$ and $\qt(\bla)=((\emp,(2),(1^3),\emp))$. 
  
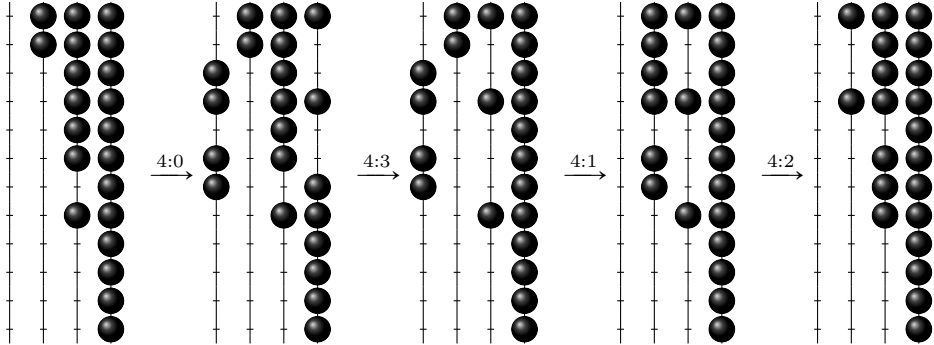
\medskip
  
  \begin{center}
\abacus(nbbb,nbbb,nnbb,nnbb,nnbb,nnbb,nnnb,nnbb,nnnb,nnnb,nnnb,nnnb) \; $\xrightarrow{4:0}$
\abacus(nbbb,nbbn,bnbn,bnbb,nnbn,bnbn,bnnb,nnbb,nnnb,nnnb,nnnb,nnnb) \;
$\xrightarrow{4:3}$
\abacus(nbbb,nbnb,bnnb,bnbb,nnnb,bnnb,bnnb,nnbb,nnnb,nnnb,nnnb,nnnb) \;
$\xrightarrow{4:1}$
\abacus(nbbb,nbnb,nbnb,nbbb,nnnb,nbnb,nbnb,nnbb,nnnb,nnnb,nnnb,nnnb) \;
$\xrightarrow{4:2}$
\abacus(nbbb,nnbb,nnbb,nbbb,nnnb,nnbb,nnbb,nnbb,nnnb,nnnb,nnnb,nnnb) \;
\end{center}
\medskip
\end{ex}

We now consider the graded version of Proposition~\ref{L:I1}. Assume that $s_{\bla}$ occurs in $f^{(s,j)}s_{\btau}$ with non-zero coefficient. Then there exist $l_0,\ldots,l_{r-1},\ell_0,\ldots,\ell_{r-1}$ with $\sum_{k=0}^{r-1} (l_k + \ell_k)=s$ and 
\[c^{\la^{k}_{j-1}}_{\tau^k_{j-1} (l_k)} = c^{\la^{k}_j}_{\tau^k_j (1^{\ell_k})} =1\]
for all $0 \leq k \leq r-1$. Hence by Lemma~\ref{L:OneS1}, $\la^{k}_{j-1}$ is formed from $\tau^{k}_{j-1}$ by adding horizontal strips in non-overlapping rows and $\la^{k}_{j}$ is formed from $\tau^{k}_{j}$ by adding vertical strips in non-overlapping columns. We have a sequence of multipartitions
\[ \btau \xrightarrow{s:0} \btau(e) \xrightarrow{s:e-1} \btau(e-1) \xrightarrow{s:e-2} \ldots \xrightarrow{s:j+2} \btau(j+2) \xrightarrow{s:j+1} \btau(j+1) \xrightarrow{s:1} \btau(1) \xrightarrow{s:2} \ldots \xrightarrow{s:j} \btau(j) = \bla.\]
Let \[N=N(\btau,\btau(e)) + \sum_{m=j+1}^{e-1} N(\btau(m+1),\btau(m)) + N(\btau(j+1),\btau(1)) + \sum_{m=2}^{j} N(\btau(m-1),\btau(m))\] 
so that the coefficient of $s_{\bla}$ in $f^{(s,j)} s_{\btau}$ is $v^N$.

Let $0 \leq k \leq r-1$. 
If $\BB_k = \BB_{a_k}(\tau^{(k)})$ and $\BB'_k = \BB_{a_k}(\la^{(k)})$ then by Lemma~\ref{L:OneS} there exist $t_k, t'_k \geq 0$ and
\begin{itemize}
\item $a^k_1<a^k_2<\ldots<a^k_{t_k}$ with $a^k_l \equiv j-1 \mod e$  and $a^k_l \in \BB_k$, $a^k_l \notin \BB'_k$ for all $l$, 
\item $b^{k}_1<b^{k}_2<\ldots<b^k_{t_k}$ with $b^k_l \equiv j-1 \mod e$  and $b^{k}_l \in \BB'_k$, $b^k_l \notin \BB_k$ for all $l$,
\item $c^k_1<c^k_2<\ldots<c^k_{t'_k}$ with $c^k_l \equiv j \mod e$  and $c^k_l \in \BB_k$, $c^k_l \notin \BB'_k$ for all $l$, 
\item $d^k_1<d^k_2<\ldots<d^k_{t'_k}$ with $d^k_l \equiv j \mod e$  and $d^k_l \in \BB'_k$, $d_l \notin \BB_k$ for all $l$,
\end{itemize}
such that 
\[a^k_1 < b^k_1 < a^k_2 < b^k_2 <\ldots < a^k_{t_k} < b^k_{t_k} < c^k_1 < d^k_1 < c^k_2 < d^k_2 < \ldots < c^k_{t'_k} < d^k_{t'_k}\]
and if $b\in \Z$ is not in the list above then $b \in \BB_k \iff b \in \BB'_k$.  (The middle inequality, $b^k_{t_k}<c^k_1$ occurs because we are in a Rouquier block.) 
We can further say that if $b \equiv j-1 \mod e$ and there exists $l$ such that $a^k_l < b \leq b^k_l$ then $b \notin \BB_k$ and if $c \equiv j \mod e$ and there exists $l$ such that $c^k_l < c \leq d^k_l$ then $c \in \BB_k$.
Let
\begin{align*}
M^k_1 &=\{g \in \Z \mid \text{there exists } 1 \leq l \leq t_k \text{ such that } a^k_l \leq ge+j-1 < b^k_l\}, 
\\
M^k_2 & = \{g \in \Z \mid \text{there exists } 1 \leq l \leq t'_k \text{ such that } c^k_l \leq ge+j < d^k_l\}. &
\end{align*}
The idea is that these sets correspond to the rows of abacus $k$ where, at various steps in the induction, we will increase a $\beta$-number. Note that
$|M^k_1| = l_k$ and $|M^k_2|=\ell_k$. 

We describe which $\beta$-numbers change at each induction step. 

\begin{itemize}
\item $\btau \xrightarrow{s:0} \btau(e)$: 
For each $0 \leq k \leq r-1$, 
 on component $k$, increase the $\beta$-number of all beads in the set $\{ge+e-1 \mid g \in M^k_1 \cup M^k_2\}$ by $1$. 
\item 
$\btau(m+1) \xrightarrow{s:m} \btau(m)$ for $e-1 \geq m \geq j+1$: For each $0 \leq k \leq r-1$,
on component $k$, increase the $\beta$-number of all beads in the set $\{ge+m-1 \mid g \in M^k_1 \cup M^k_2\}$ by $1$. 
\item 
$\btau(j+1) \xrightarrow{s:1} \btau(1)$: For each $0 \leq k \leq r-1$,
on component $k$, increase the $\beta$-number of all beads in the set $\{(g+1)e \mid g \in M^{k}_1 \cup M^{k}_2\}$ by $1$.
\item 
$\btau(m-1) \xrightarrow{s:m} \btau(m)$ for $1 \leq m \leq j-1$: For each $0 \leq k \leq r-1$,
on component $k$, increase the $\beta$-number of all beads in the set $\{(g+1)e+m-1 \mid g \in M^{k}_1 \cup M^{k}_2\}$ by $1$.
\item $\btau(j-1) \xrightarrow{s:j} \btau(j)$:  For each $0 \leq k \leq r-1$,
on component $k$, increase the $\beta$-number of all beads in the set 
$\{ge+j-1 \mid g \in M^k_1\} \cup \{(g+1)e +j-1 \mid g \in M^{k}_2\}$ by $1$. 
\end{itemize}
So for $g \in \bigcup_{k=0}^{r-1} M^k_1 \cup M^k_2$, we can consider the contribution made by $g$ to $N$. First, consider the removable and addable nodes from components $l<k$. 
At the first step of the induction, there are addable $0$-nodes on component $l$ of $\bnu$ but no removable $0$-nodes on component $l$ of $\bla$. At later steps, only removable nodes contribute to the sum. 
The contribution is 
\[\sum_{l=0}^{k-1} \Big((\mathfrak{b}^{l}_{e-1}(\btau) - \mathfrak{b}^{l}_0(\btau) +1) 
- \sum_{m=1}^{e-1} (\mathfrak{b}^{l}_{m}(\btau) - \mathfrak{b}^{l}_{m-1}(\btau)) \Big) = k,\]
where we recall from Section~\ref{SS:Rouq} that $\mathfrak{b}^{l}_i(\btau)$ is the position of the last bead on runner $i$ of abacus $l$ in the abacus configuration for $\bar{\btau}$. 
Now consider the removable and addable nodes from component $k$ which contribute to the sum. If $g \in M_2^k$, then over the course of the inductions, there is an overall contribution of $0$. If $g \in M_1^k$, then over the course of the inductions, there is an overall contribution of $1$. Hence 
\[N= \sum_{k=0}^{r-1}\left( \sum_{g \in M^k_1} (k+1) + \sum_{g \in M^k_2} k \right) = \sum_{k=0}^{r-1} (k+1) l_k + k \ell_k.\] 
\end{proof}

There are two immediate consequences of Proposition~\ref{L:I1} which will be important later. 

\begin{corollary}
Suppose $s >0$ and that $\btau \in \Rs$. Let $1 \leq j \leq e-1$. 
If $f^{(s,j)} s_{\btau} = \sum_{\bla} b_{\bla}(v) s_{\bla}$, then $b_{\bla}(v) \in \N[v]$ for all $\bla \in \La^r$.   
\end{corollary}

\begin{corollary}
Suppose $s >0$ and that $\btau, \bsig \in \Rs$ are such that $\btau \sim \bsig$. Let $1 \leq j \leq e-1$. Suppose that $s_{\bmu}$ appears with non-zero coefficient in $f^{(s,j)} s_{\btau}$ and $s_{\bla}$ appears with non-zero coefficient in $f^{(s,j)} s_{\bsig}$. Then $\btau \approx \bsig$ if and only if $\bmu \approx \bla$. 
\end{corollary}

Now suppose that $\bmu \in \RR$. Define
\[Q(\bmu) = \sum_{\bla \approx \bmu} g_{\bla\bmu}(v) s_{\bla} \qquad \text{ and } \qquad \check{Q}(\bmu) = \sum_{\bla \approx \bmu} g_{\bla\bmu}(1) s_{\bla} .\] 

The heavy lifting for the proof of the main theorem is done by Proposition~\ref{P:HeavyGraded} below. We first prove an ungraded version of the proposition. We abuse notation slightly by identifying a multipartition with its quotient. 

When $r=1$, Proposition~\ref{P:HeavyUngraded} appears as~\cite[Lemma~3.10 b]{JLM}. 

\begin{proposition} \label{P:HeavyUngraded}
Suppose $s >0$ and that $\bnu \in \Rrs$. 
Let $1 \leq j \leq e-1$.
Then
  \[f^{(s,j)} \check{Q}(\bnu) = \sum_{\Delta \in \La}c^{\Delta}_{\nu^{0}_j (1^s)} \check{Q}((\nu^0_0,\ldots,\nu^{0}_{j-1},\Delta,\nu^{0}_{j+1},\ldots,\nu^0_{e-1}),(\nu^1_0,\nu^1_1,\ldots,\nu^1_{e-1}),\ldots,(\nu^{r-1}_0,\nu^{r-1}_1,\ldots,\nu^{r-1}_{e-1})). \]
  \end{proposition}

\begin{proof}
We look for the coefficient of the basis element $s_{\bla} \in \R$ on both sides of the equation. By Proposition~\ref{L:I1}, on the left the coefficient is given by
\begin{align} \label{E:LeftSide} \nonumber
\sum_{\btau \approx \bnu} & \sum_{\ba \in \M{r}{e+1}} \sum_{\bb \in \M{r}{e}} \sum_{\bc \in \M{{r+1}}{e}} \sum_{\bd \in \M{r}{e}} \left( \prod_{k=0}^{r-1} \prod_{i=0}^{e-1} c^{d^k_i}_{\nu^k_i c^k_i} c^{d^k_i}_{a^k_i b^k_i c^{k+1}_i } c^{\tau^k_i}_{b^k_i (a^k_{i+1})'}\right)  c^{\emp}_{c^0_0 c^0_1 \ldots c^{0}_{e-1}} c^{\emp}_{c^r_0 c^r_1 \ldots c^r_{e-1}}\\ \nonumber
&\qquad \left(\prod_{k=0}^{r-1} \prod_{i \neq j-1,j} c^{\la^{k}_i}_{\tau^k_i}\right) 
\left( \sum_{\substack{x_0+\ldots+x_{r-1} \\ + y_0+\ldots+y_{r-1} = s}} \prod_{k=0}^{r-1} c^{\la^k_{j-1}}_{\tau^{k}_{j-1} (x_k)} c^{\la^{k}_{j}}_{\tau^k_j (1^{y_k})}\right)\\ \nonumber
& = \sum_{\substack{\tau^0_{j-1},\ldots,\tau^{r-1}_{j-1} \\ \tau^0_j,\ldots,\tau^{r-1}_j}} \sum_{\ba \in \M{r}{e+1}} \sum_{\bb \in \M{r}{e}} \sum_{\bc \in \M{{r+1}}{e}} \sum_{\bd \in \M{r}{e}} \left( \prod_{k=0}^{r-1} \prod_{i \neq j-1,j} c^{d^k_i}_{\nu^k_i c^k_i} c^{d^k_i}_{a^k_i b^k_i c^{k+1}_i } c^{\la^k_i}_{b^k_i (a^k_{i+1})'}\right) c^{\emp}_{c^0_0 c^0_1 \ldots c^{0}_{e-1}} c^{\emp}_{c^r_0 c^r_1 \ldots c^{r}_{e-1}} \\ 
&\qquad 
\sum_{\substack{x_0+\ldots+x_{r-1} \\ + y_0+\ldots+y_{r-1} = s}} \prod_{k=0}^{r-1}
c^{d^k_{j-1}}_{\nu^k_{j-1} c^k_{j-1}} c^{d^k_{j-1}}_{a^k_{j-1} b^k_{j-1} c^{k+1}_{j-1}} c^{\tau^k_{j-1}}_{b^k_{j-1} (a^k_{j})'} 
c^{d^k_j}_{\nu^k_j c^k_j} c^{d^k_j}_{a^k_j b^k_j c^{k+1}_j} c^{\tau^k_j}_{b^k_j (a^k_{j+1})'} c^{\la^k_{j-1}}_{\tau^{k}_{j-1} (x_k)} c^{\la^{k}_{j}}_{\tau^k_j (1^{y_k})}
\end{align}
while on the right it is given by
\begin{align} \label{E:RightSide}
\sum_{\Delta \in \La} c^{\Delta}_{\nu^0_j (1^s)} 
& \sum_{\bal \in \M{r}{e+1}} \sum_{\bbe \in \M{r}{e}} \sum_{\bgam \in \M{{r+1}}{e}} \sum_{\bdel \in \M{r}{e}} \left(\prod_{i \neq j} c^{\delta^0_i}_{\nu^0_i \gamma^0_i} c^{\delta^0_i}_{\al^0_i \be^0_i \gamma^{1}_i } c^{\la^0_i}_{\be^0_i (\al^0_{i+1})'}\right) c^{\delta^0_j}_{\Delta \gamma^0_j} c^{\delta^0_j}_{\al^0_j \be^0_j \gamma^{1}_j } c^{\la^0_j}_{\be^0_j (\al^0_{j+1})'} \nonumber\\ 
& \qquad \qquad \left( \prod_{k=1}^{r-1} \prod_{i=0}^{e-1} c^{\delta^k_i}_{\nu^k_i \gamma^k_i} c^{\delta^k_i}_{ \al^k_i \be^k_i \gamma^{k+1}_i} c^{\la^k_i}_{\be^k_i (\al^k_{i+1})'}\right) c^{\emp}_{\gamma^0_0 \gamma^0_1 \ldots \gamma^0_{e-1}} c^{\emp}_{\gamma^r_0 \gamma^r_1 \ldots \gamma^r_{e-1}}  \nonumber \\
& = \sum_{\bal \in \M{r}{e+1}} \sum_{\bbe \in \M{r}{e}} \sum_{\bgam \in \M{{r+1}}{e}} \sum_{\bdel \in \M{r}{e}} \left(\prod_{i \neq j} c^{\delta^0_i}_{\nu^0_i \gamma^0_i} c^{\delta^0_i}_{ \al^0_i \be^0_i \gamma^{1}_i} c^{\la^0_i}_{\be^0_i (\al^0_{i+1})'}\right) c^{\delta^0_j}_{\nu^0_j \gamma^0_j (1^s)} c^{\delta^0_j}_{ \al^0_j \be^0_j \gamma^{1}_j} c^{\la^0_j}_{\be^0_j (\al^k_{j+1})'} \nonumber \\ 
& \qquad \qquad \left( \prod_{k=1}^{r-1} \prod_{i=0}^{e-1} c^{\delta^k_i}_{\nu^k_i \gamma^k_i} c^{\delta^k_i}_{\al^k_i \be^k_i \gamma^{k+1}_i } c^{\la^k_i}_{\be^k_i (\al^k_{i+1})'}\right) c^{\emp}_{\gamma^0_0 \gamma^0_1 \ldots \gamma^0_{e-1}} c^{\emp}_{\gamma^r_0 \gamma^r_1 \ldots \gamma^r_{e-1}} 
\end{align}

 Now we look at some of the terms from the second line of Equation~\ref{E:LeftSide}. For fixed partitions $d^k_j$, $c^{k+1}_j$, $a^k_{j+1}$, $b^k_{j-1}$, $\tau^{k}_{j-1}$, $\tau^k_j$ where $0 \leq k \leq r-1$ we have
 \begin{align*} 
\sum_{a^0_j,\ldots,a^{r-1}_j} & \sum_{b^0_j,\ldots,b^{r-1}_j}  
\sum_{\substack{x_0+\ldots+x_{r-1} \\ + y_0+\ldots+y_{r-1} = s}} \prod_{k=0}^{r-1}
c^{\tau^k_{j-1}}_{b^k_{j-1} (a^k_{j})'} 
c^{d^k_j}_{a^k_j b^k_j c^{k+1}_j } c^{\tau^k_j}_{b^k_j (a^k_{j+1})'} c^{\la^k_{j-1}}_{\tau^{k}_{j-1} (x_k)} c^{\la^{k}_{j}}_{\tau^k_j (1^{y_k})} \\
&= \sum_{a^0_j,\ldots,a^{r-1}_j} \sum_{b^0_j,\ldots,b^{r-1}_j}  
\sum_{\substack{x_0+\ldots+x_{r-1} \\ + y_0+\ldots+y_{r-1} = s}} \prod_{k=0}^{r-1}
c^{d^k_j}_{a^k_j b^k_j c^{k+1}_j } c^{\tau^k_{j-1}}_{b^k_{j-1} (a^k_{j})'} c^{\la^k_{j-1}}_{\tau^{k}_{j-1} (x_k)} c^{\tau^k_j}_{b^k_j (a^k_{j+1})'}
c^{\la^{k}_{j}}_{\tau^k_j (1^{y_k})} \\
\intertext{and applying Lemma~\ref{LR1} repeatedly}
&=\sum_{a^0_j,\ldots,a^{r-1}_j} \sum_{b^0_j,\ldots,b^{r-1}_j} \sum_{e^0,\ldots,e^{r-1}}  \sum_{\substack{x_0+\ldots+x_{r-1} \\ + y_0+\ldots+y_{r-1} = s}} \prod_{k=0}^{r-1}
c^{d^k_j}_{c^{k+1}_j e^k} c^{e^k}_{a^k_j b^k_j} c^{\la_{j-1}^k}_{b^k_{j-1}\tau^k_{j-1}}c^{\tau^k_{j-1}}_{(a^k_j)'(x_k)}  c^{\la^{k}_j}_{(a^k_{j+1})' \tau^k_j}  c^{\tau^k_j}_{b^k_j (1^y_k)} \\ 
&=\sum_{a^0_j,\ldots,a^{r-1}_j} \sum_{b^0_j,\ldots,b^{r-1}_j} \sum_{e^0,\ldots,e^{r-1}}  \sum_{\substack{x_0+\ldots+x_{r-1} \\ + y_0+\ldots+y_{r-1} = s}} \prod_{k=0}^{r-1}
c^{d^k_j}_{c^{k+1}_j e^k}  c^{\la_{j-1}^k}_{b^k_{j-1}\tau^k_{j-1}} c^{\la^{k}_j}_{(a^k_{j+1})' \tau^k_j} c^{e^k}_{a^k_j b^k_j}c^{\tau^k_{j-1}}_{(a^k_j)'(x_k)}  c^{\tau^k_j}_{b^k_j (1^y_k)} \\ 
\intertext{and applying Lemma~\ref{LR2}} 
 &= \sum_{e^0,\ldots,e^{r-1}} \sum_{f^0,\ldots,f^{r-1}} \sum_{z_1+\ldots+z_r = s} \prod_{k=0}^{r-1}
c^{d^k_j}_{c^{k+1}_j e^k} c^{\la_{j-1}^k}_{b^k_{j-1}\tau^k_{j-1}}
c^{\la^{k}_j}_{(a^k_{j+1})' \tau^k_j}
c^{f_k}_{(\tau^k_{j-1})' \tau^k_j} c^{f_k}_{e^k (1^{z_k})} \\
\end{align*}
 
We now substitute this expression back into Equation~\ref{E:LeftSide}. We also perform the change of variables $\tau^k_{j-1} \mapsto (a^k_j)'$ and $\tau^k_j \mapsto b^k_j$ for $0 \leq k \leq r-1$. Thus we see that Equation~\ref{E:LeftSide} is equal to
\begin{align} \label{E:Next}
\sum_{\ba \in \M{r}{e+1}} & \sum_{\bb \in \M{r}{e}} \sum_{\bc \in \M{{r+1}}{e}} \sum_{\bd \in \M{r}{e}} \left( \prod_{k=0}^{r-1} \prod_{i \neq j-1,j} c^{d^k_i}_{\nu^k_i c^k_i} c^{d^k_i}_{ a^k_i b^k_i c^{k+1}_i} c^{\la^k_i}_{b^k_i (a^k_{i+1})'}\right) c^{\emp}_{c^0_0 c^0_1 \ldots c^0_{e-1}} c^{\emp}_{c^r_0 c^r_1 \ldots c^r_{e-1}}  \nonumber\\ \nonumber
& \qquad \sum_{e^0,\ldots,e^{r-1}} \sum_{f^0,\ldots,f^{r-1}} \sum_{\substack{z_0+\ldots+z_{r-1} = s}} \prod_{k=0}^{r-1}
c^{d^k_{j-1}}_{\nu^k_{j-1} c^k_{j-1}} c^{d^k_{j-1}}_{ a^k_{j-1} b^k_{j-1} c^{k+1}_{j-1}} 
c^{d^k_j}_{\nu^k_j c^k_j} 
c^{d^k_j}_{c^{k+1}_j e^k} c^{\la_{j-1}^k}_{b^k_{j-1}(a^k_{j})'}
c^{\la^{k}_j}_{b^k_j (a^k_{j+1})'} c^{f_k}_{a^k_{j} b^k_j}
c^{f_k}_{e^k(1^{z_k})} \\ \nonumber
&= \sum_{\ba \in \M{r}{e+1}} \sum_{\bb \in \M{r}{e}} \sum_{\bc \in \M{{r+1}}{e}} \sum_{\bd \in \M{r}{e}} \left( \prod_{k=0}^{r-1} \prod_{i \neq j-1,j} c^{d^k_i}_{\nu^k_i c^k_i} c^{d^k_i}_{ a^k_i b^k_i c^{k+1}_i} c^{\la^k_i}_{b^k_i (a^k_{i+1})'}\right)  c^{\emp}_{c^0_0 c^0_1 \ldots c^{0}_{e-1}} \\ \nonumber
& \qquad \left(\prod_{k=0}^{r-1} 
c^{d^k_{j-1}}_{\nu^k_{j-1} c^k_{j-1}} c^{d^k_{j-1}}_{ a^k_{j-1} b^k_{j-1} c^{k+1}_{j-1}}
c^{\la_{j-1}^k}_{b^k_{j-1}(a^k_{j})'} \right) \\
& \qquad \qquad \left(\sum_{e^0,\ldots,e^{r-1}} \sum_{f^0,\ldots,f^{r-1}} \sum_{\substack{z_0+\ldots+z_{r-1} = s}} \prod_{k=0}^{r-1}c^{d^k_j}_{\nu^k_j c^k_j} 
c^{d^k_j}_{c^{k+1}_j e^k}c^{\la^{k}_j}_{b^k_j (a^k_{j+1})'} c^{f^k}_{a^k_{j} b^k_j}
c^{f^k}_{e^k(1^{z_k})}
\right) c^{\emp}_{c^r_0 c^r_1 \ldots c^r_{e-1}}
\end{align}

Now we focus on the last line of Equation~\ref{E:Next}. Applying Lemma~\ref{L:Tunnel} we have
\begin{align*}
\sum_{c^{0}_j,\ldots,c^{r}_j} & \sum_{d^0_j,\ldots,d^{r-1}_j}\sum_{e^0,\ldots,e^{r-1}}  \sum_{f^0,\ldots,f^{r-1}} \sum_{\substack{z_0+\ldots+z_{r-1} = s}} \prod_{k=0}^{r-1} c^{f^k}_{a^k_{j} b^k_j}c^{\la^{k}_j}_{b^k_j (a^k_{j+1})'} c^{d^k_j}_{\nu^k_j c^k_j} 
c^{d^k_j}_{c^{k+1}_j e^k}
c^{f^k}_{e^k(1^{z_k})} c^{\emp}_{c^r_0 c^r_1 \ldots c^r_{e-1}} \\
& = \sum_{c^{0}_j,\ldots,c^{r}_j}\sum_{d^0_j,\ldots,d^{r-1}_j} \sum_{f^0,\ldots,f^{r-1}} \left(\prod_{k=0}^{r-1} 
c^{f^k}_{a^k_{j} b^k_j}c^{\la^{k}_j}_{b^k_j (a^k_{j+1})'}\right) 
c^{d^0_j}_{\nu^0_j c^0_j (1^s)}
\left(\prod_{k=1}^{r-1} c^{d^{k-1}_j}_{c^k_j f^{k-1}} c^{d^k_j}_{\nu^k_j c^k_j} \right)c^{f^{r-1}}_{d^{r-1}_j} 
c^{\emp}_{c^r_0 c^r_1 \ldots c^r_{e-1}} \\
& = \sum_{c^{0}_j,\ldots,c^{r}_j}\sum_{d^0_j,\ldots,d^{r-1}_j} \sum_{f^0,\ldots,f^{r-1}} \left(\prod_{k=0}^{r-1} 
c^{f^k}_{a^k_{j} b^k_j}c^{\la^{k}_j}_{b^k_j (a^k_{j+1})'}  c^{d^{k}_j}_{c^{k+1}_j f^{k}}\right) 
c^{d^0_j}_{\nu^0_j c^0_j (1^s)}
\left(\prod_{k=1}^{r-1} c^{d^k_j}_{\nu^k_j c^k_j} \right)
c^{\emp}_{c^r_0c^r_1 \ldots c^r_{e-1}} \\
& = \sum_{c^{0}_j,\ldots,c^{r}_j}\sum_{d^0_j,\ldots,d^{r-1}_j}
c^{d^0_j}_{a^0_j b^0_j c^1_j} c^{\la^0_j}_{b^0_j (a^0_{j+1})'} c^{d^0_j}_{\nu^0_j c^0_j (1^s)} 
\left(\prod_{k=1}^{r-1} c^{d^{k}_j}_{a^k_j b^k_j c^{k+1}_j} c^{\la^k_j}_{b^k_j (a^k_{j+1})'} c^{d^k_j}_{\nu^k_j c^k_j}\right) c^{\emp}_{c^r_0 c^r_1 \ldots c^r_{e-1}}
\end{align*}

Finally, substituting this expression back into Equation~\ref{E:Next} we see that Equation~\ref{E:LeftSide} is equal to 
\begin{multline*}
 \sum_{\ba \in \M{r}{e+1}} \sum_{\bb \in \M{r}{e}} \sum_{\bc \in \M{{r+1}}{e}} \sum_{\bd \in \M{r}{e}} 
\left( \prod_{i \neq j} c^{d^0_j}_{\nu^0_i c^0_i} c^{d^0_i}_{c^1_i a^0_i b^0_i} c^{\la^0_i}_{b^0_i (a^0_{i+1})'}
\right) 
c^{d^0_j}_{a^0_j b^0_j c^1_j} c^{\la^0_j}_{b^0_j (a^0_{j+1})'} c^{d^0_j}_{\nu^0_j c^0_j (1^s)}\\
 \left( \prod_{k=1}^{r-1} \prod_{i=0}^{e-1} c^{d^k_i}_{\nu^k_i c^k_i} c^{d^k_i}_{a^k_i b^k_i c^{k+1}_i } c^{\la^k_i}_{b^k_i (a^k_{i+1})'}\right)  c^{\emp}_{c^0_0 c^0_1 \ldots c^{0}_{e-1}} c^{\emp}_{c^r_0 c^r_1 \ldots c^r_{e-1}}
\end{multline*}
This is indeed equal to Equation~\ref{E:RightSide}. 
\end{proof}

\begin{proposition} \label{P:HeavyGraded}
Suppose $s >0$ and that $\bnu \in \Rrs$.
Let $1 \leq j \leq e-1$.
Then
  \begin{equation} \label{E:Grad}
 f^{(s,j)} Q(\bnu) = \sum_{\Delta}c^{\Delta}_{\nu^{0}_j (1^s)} Q((\nu^0_0,\ldots,\nu^{0}_{j-1},\Delta,\nu^{0}_{j+1},\ldots,\nu^0_{e-1}),(\nu^1_0,\nu^1_1,\ldots,\nu^1_{e-1}),\ldots,(\nu^{r-1}_0,\nu^{r-1}_1,\ldots,\nu^{r-1}_{e-1})). \end{equation}
\end{proposition}

\begin{proof}
By Proposition~\ref{P:HeavyUngraded}, if we set $v=1$ in Equation~\ref{E:Grad}, both sides of the equation are equal to $\sum_{\bla} x_{\bla} s_{\bla}$ for some $x_{\bla} \in \Z$. 
We claim that both sides of Equation~\ref{E:Grad} are then equal to 
\begin{equation} \label{E:L1} 
\sum_{\bla} v^{\omega(\bla)-\omega(\bnu)-(e-j-1)s} x_{\bla} s_{\bla}.
\end{equation}
As before, we fix $\bla$ and look for the coefficient of $s_{\bla}$ on both sides of the equation. 

Let $\bDel$ be any multipartition with 
\[\qt(\bDel) = ((\nu^0_0,\ldots,\nu^{0}_{j-1},\Delta,\nu^{0}_{j+1},\ldots,\nu^0_{e-1}),(\nu^1_0,\nu^1_1,\ldots,\nu^1_{e-1}),\ldots,(\nu^{r-1}_0,\nu^{r-1}_1,\ldots,\nu^{r-1}_{e-1}))\]
where $c^{\Delta}_{\nu^0_j (1^s)}>0$. 
Then the coefficient of $s_{\bla}$ on the right hand side of Equation~\ref{E:Grad} is equal to $v^{\omega(\bla) - \omega(\bDel)} x_{\bla}$ where $\omega(\bDel) = \omega(\bnu) + (e-j-1)s$.
So the right hand side of Equation~\ref{E:Grad} is indeed equal to Equation~\ref{E:L1}. 

Now suppose that $\btau \approx \bnu$. Then $g_{\btau \bnu}(v) = v^{\omega(\btau)-\omega(\bnu)} g_{\btau \bnu}(1)$. If $s_{\bla}$ appears in $f^{(s,j)} s_{\btau}$, it appears with coefficient $v^M$ where
\[M=\sum_{k=0}^{r-1} (k+1)(|\la^k_{j-1}| - |\tau^k_{j-1}|) + k(|\la^k_j| - |\tau^k_j|).\]
Thus we have a contribution to the sum of $v^{M'}g_{\btau\bnu}(1)s_{\bla}$ where
\begin{equation*} \label{E:EE1} 
M'=\omega(\btau) - \omega(\bnu) + \sum_{k=0}^{r-1} (k+1)(|\la^k_{j-1}| - |\tau^k_{j-1}|) + k(|\la^k_j| - |\tau^k_j|). \end{equation*}
Now note that 
\begin{align*}
\omega(\bla)-\omega(\btau) & = \sum_{k=0}^{r-1} (e-j+k)(|\la^k_{j-1}|-|\tau^k_{j-1}|) + (e-j+k-1)(|\la^k_j|-|\tau^k_j|) \\
& = \sum_{k=0}^{r-1}(e-j-1)(|\la^k_{j-1}|-|\tau^k_{j-1}| + |\la^k_j| - |\tau^k_j|) + (k+1)(|\la^k_{j-1}| - |\tau^k_{j-1}|) + k(|\la^k_j| - |\tau^k_j|) \\
&= (e-j-1)s +  \sum_{k=0}^{r-1}(k+1)(|\la^k_{j-1}| - |\tau^k_{j-1}|) + k(|\la^k_j| - |\tau^k_j|).
\end{align*}
Hence
\[M'=\omega(\bla)-\omega(\bnu)-(e-j-1)s,\]
which is independent of the choice of $\btau$. 
So the right hand side of Equation~\ref{E:Grad} is indeed equal to Equation~\ref{E:L1}. 
\end{proof}

\subsection{Proof of the main results} \label{SS:ProofMain}

\begin{theorem} \label{T:Main}
Suppose that $(\bla,\bmu) \in \Rp$. Then
  \[d_{\bla\bmu}(v) = g_{\bla\bmu}(v).\]
  \end{theorem}

\begin{proof}
Let $\bmu \in \RR$. We want to show that
\begin{equation} \label{E:Prove}
G(\bmu) = \sum_{\bla \approx \bmu} g_{\bla\bmu}(v)s_{\bla} + \sum_{\substack{\bla \sim \bmu \\ \bla \not \approx \bmu}} d_{\bla\bmu}(v) s_{\bla}.\end{equation}
We prove this by induction, firstly on $r$, secondly on $\hook(\mu^{(0)})$ and thirdly using the a total order $\succ$ on $\mu^{(0)}$. 
Suppose that $r=1$. Then Equation~\ref{E:Prove} holds by Lemma~\ref{L:r1} and~\cite[Corollary 10]{LeclercMiyachi2}. So suppose that $r>1$, that Theorem~\ref{T:Main} holds for $r-1$ and that $\bmu \in \RR$. Suppose first that $\hook(\mu^{(0)})=0$.  
Following Fayers~\cite{Fayers:LLT}, if $\bnu=(\nu^{(0)},\nu^{(1)},\ldots,\nu^{(r-1)}) \in \La^r$, we define 
$\hat{\bnu} = (\nu^{(1)},\ldots,\nu^{(r-1)})$ and $\hat{\bnu}^{\emp} = (\emp,\nu^{(1)},\ldots,\nu^{(r-1)})$.
Then $\hat{\bmu} \in \RR$ and by the inductive hypothesis
\[G(\hat{\bmu})= \sum_{\bla \approx \hat{\bmu}} g_{\bla \hat{\bmu}}(v) s_{\bla} +
\sum_{\substack{\bla \sim \hat{\bmu} \\ \bla \not \approx \hat{\bmu}}}
d_{\bla \hat{\bmu}}(v) s_{\bla}.\] 
Applying Lemma~\ref{L:AddEmpty}, 
 \[G(\hat{\bmu}^{\emp})= \sum_{\substack{\bla \approx \hat{\mu}^{\emp}\\ \la^{(0)}=\emp}} g_{\hat{\bla} \hat{\bmu}}(v) s_{\bla} +
 \sum_{\substack{\bla \sim \hat{\bmu}^{\emp}, \bla \not \approx \hat{\bmu}^{\emp} \\ \la^{(0)} =\emp}}
 d_{\hat{\bla} \hat{\bmu}}(v) s_{\bla}.\] 
Now $\bmu^{(0)}$ is $e$-regular and is an $e$-core. By Lemma~\ref{L:LLTLead}, there exists $f = f^{(t_x)}_{i_x} f^{(t_{x-1})}_{i_{x-1}} \ldots f^{(t_1)}_{i_1} \in \U$ such that 
\[f (G(\hat{\bmu}^{\emp})) = \sum_{\substack{\bla \approx \bmu \\ \la^{(0)}=\mu^{(0)}}} g_{\bla\bmu}(v) s_{\bla} +
 \sum_{\substack{\bla \sim \bmu, \bla \not \approx \bmu \\ \la^{(0)} =\mu^{(0)}}}
 d_{\hat{\bla} \hat{\bmu}}(v) s_{\bla}
 + \sum_{\substack{\btau \sim \bmu \\ |\tau^{(0)}|<|\mu^{(0)}|}} b_{\btau}(v)s_{\btau}\]
 for some $b_{\btau}(v) \in \N[v^{-1},v]$, where we note that the identities 
 $g_{\hat{\bla} \hat{\bmu}}(v)=g_{\bla\bmu}(v)$ follow from Lemma~\ref{L:MinusOK}. Now $f(G(\hat{\bmu}^{\emp})) \in \U^{\circ}$ and of the coefficients in the sum, $g_{\bla\bmu}(v) \in v\N[v]$ for $\bla \neq \bmu$ and $d_{\hat{\bla}\hat{\bmu}}(v) \in v\N[v]$. It is possible that we have $b_{\btau}(v) \notin v\N[v]$, but this is of no consequence since by
Lemma~\ref{L:DomNeeded}, $d_{\bla\btau}(v)=0$ if $|\tau^{(0)}|<|\la^{(0)}|$.  
Thus we can conclude that
\[f (G(\hat{\bmu}^{\emp})) = G(\bmu) + \sum_{\substack{\btau \sim \bmu \\ |\tau^{(0)}|<|\mu^{(0)}|}} c_{\btau} G(\btau)\]
for some $c_{\btau} \in \N[v^{-1}+v]$, and equating terms we see that 
\[G(\bmu) =  \sum_{\bla \approx \bmu} g_{\bla\bmu}(v)s_{\bla} + \sum_{\substack{\bla \sim \bmu \\ \bla \not \approx \bmu}} d_{\bla\bmu}(v) s_{\bla}\]
as required.

This completes the proof of Theorem~\ref{T:Main} for multipartitions $\bmu$ when $\hook(\mu^{(0)})=0$. So now suppose that $h=\hook(\mu^{(0)})>0$ and that Theorem~\ref{T:Main} holds for all multipartitions in $\RR$ with fewer than $h$ removable $e$-rim hooks on the first component. Let $H(\bmu)$ be the set of partitions $\bla \approx \bmu$ such that $|\la^k_i|=|\mu^k_i|$ for all $0 \le i \le e-1$ and $0 \leq k \leq r-1$ and define a total order $\succ$ on $H(\bmu)$ by saying that $\bla \succ {\bf \Delta}$ if $\bla \neq {\bf \Delta}$ and the minimal $k$ such that $\la^{(k)} \neq \Delta^{(k)}$, the minimal $i$ such that $\la^{k}_i \neq \Delta^{k}_i$ and the minimal $x$ such that $(\la^k_i)_x \neq (\Delta^k_i)_x$ satisfy $(\la_i^{(k)})_x > (\Delta_{i}^{(k)})_x$. 
Assume that the inductive hypothesis holds for all multipartitions ${\bf \Delta} \in H(\bmu)$ where $\bmu \succ {\bf \Delta}$. 
  
Choose $1 \leq j \leq e-1$ such that $\mu^0_j \neq \emp$. Let $\nu$ be the partition whose Young diagram is obtained by removing the first column from $[\mu^0_j]$, where we suppose that we remove $s>1$ nodes. Let $\bnu$ be the multipartition with the same multicore as $\bmu$ and
\[\nu^k_i = \begin{cases} \nu, & k=0 \text{ and } i=j, \\
  \mu^k_i, & \text{otherwise}.\end{cases}\]
By the inductive hypothesis,
\[G(\bnu)= \sum_{\bla \approx \bnu} g_{\bla\bnu}(v) s_{\bla} +
\sum_{\substack{\bla \sim \bnu \\ \bla \not \approx \bnu}} d_{\bla\bnu}(v) s_{\bla} = Q(\bnu) + \sum_{\substack{\bla \sim \bnu \\ \bla \not \approx \bnu}} d_{\bla\bnu}(v) s_{\bla}.\]
Now 
\begin{equation} \label{E:D1} 
f^{(s,j)} G(\bnu)
= f^{(s,j)} Q(\bnu) +
f^{(s,j)} \sum_{\substack{\bla \sim \bnu \\ \bla \not \approx \bnu}} d_{\bla\bnu}(v) s_{\bla}
 = \sum_{\Delta}c^{\Delta}_{\nu (1^s)} Q({\bf \Delta}) + f^{(s,j)} \sum_{\substack{\bla \sim \bnu \\ \bla \not \approx \bnu}} d_{\bla\bnu}(v) s_{\bla}
\end{equation}
by Proposition~\ref{P:HeavyGraded}, where we use the terminology of  that proposition so that 
\[\qt({\bf \Delta}) = ((\nu^0_0,\ldots,\nu^{0}_{j-1},\Delta,\nu^{0}_{j+1},\ldots,\nu^0_{e-1}),(\nu^1_0,\nu^1_1,\ldots,\nu^1_{e-1}),\ldots,(\nu^{r-1}_0,\nu^{r-1}_1,\ldots,\nu^{r-1}_{e-1})).\]
Now note that if $\bla \sim \bnu$ but $\bla \not \approx \bnu$ then $d_{\bla\bmu}(v) \in v \mathbb{Z}[v]$ from the definition of the canonical basis. Hence by Proposition~\ref{L:I1}, we see that
\[f^{(s,j)} \sum_{\substack{\bla \sim \bnu \\ \bla \not \approx \bnu}} d_{\bla\bnu}(v) s_{\bla} = \sum_{\substack{\btau \sim \bmu \\ \btau \not \approx \bmu}} b_{\btau}(v) s_{\btau}\]
where $b_{\btau}(v) \in v \mathbb{Z}[v]$ for all $\btau$. (We note that for this step it is necessary to have $\bmu \in \RR$ rather than just $\bmu$ an $e$-regular Rouquier multipartition, since we need to apply Proposition~\ref{L:I1} to multipartitions $\btau \sim \bmu$ with $\btau \not \approx \bmu$.) It follows by Corollary~\ref{C:Ones} that 
\begin{equation} \label{E:D2}
f^{(s,j)} G(\bnu) = \sum_{\Delta} c^{\Delta}_{\nu (1^s)} G({\bf \Delta})  
= G(\bmu) + \sum_{\bmu \succ {\bf \Delta}} c^{\Delta}_{\nu (1^s)} G({\bf \Delta})
\end{equation}
where the last equality follows from Lemma~\ref{L:OneS1}. 
Now suppose $\bla \approx \bmu$ and consider the coefficient of $s_{\bla}$ in $f^{(s,j)}G(\bnu)$. By Equation~\ref{E:D2}, this coefficient is given by
\[d_{\bla\bmu}(v) + \sum_{\bmu \succ {\bf \Delta}} c^{\Delta}_{\nu(1^s)} d_{{\bf \Delta} \bmu}(v) = d_{\bla\bmu}(v) + \sum_{\bmu \succ {\bf \Delta}} c^{\Delta}_{\nu(1^s)} g_{{\bf \Delta} \bmu}(v)\]
where the last step follows from the inductive hypothesis. However, by Equation~\ref{E:D1}, the coefficient is given by
\[g_{\bla\bmu}(v) + \sum_{\Delta} c^{\Delta}_{\nu^0_j(1^s)} g_{{\bf \Delta} \bmu}(v)\]
and so we have $g_{\bla\bmu}(v)=d_{\bla\bmu}(v)$ as required. By induction, this completes the proof of Theorem~\ref{T:Main}.
\end{proof}

The next theorem follows immediately by applying Theorem~\ref{T:Ariki} to Theorem~\ref{T:Main}.

\begin{theorem} \label{T:Main2}
Suppose that $\h_{r,n}(q,{\bf Q})$ is defined over a field of characteristic $0$. Take $\bla \in \La^r_n$ and $\bmu \in \Kl$ with $(\bla,\bmu) \in \Rp$. Then 
  \[[S^{\bla}:D^{\bmu}]_v = g_{\bla\bmu}(v).\]
\end{theorem}

\subsection{The cyclotomic $q$-Schur algebra and characteristic $p$} \label{SS:Schur}
In order to look at the case where the underlying field $\mathbb{F}$ has prime characteristic, we introduce a new player, the cyclotomic $q$-Schur algebra. For the definition of this algebra and the construction of the Weyl modules, we refer the reader to~\cite[Section~4]{Mathas:AKSurvey} and for the graded theory we refer them to~\cite{SW}.

Fix $r \geq 1$ and $e \geq 2$. Let $I=\{0,1,\ldots,e-1\}$ and let $\mc \in I^r$. Suppose that $\mathbb{F}$ is a field of characteristic $p\ge 0$. For $n \geq0$, take $\h=\h_{r,n}(q,{\bf Q})$ to be an Ariki-Koike algebra over a field of characteristic $p$ with quantum characterstic $e$ and multicharge $\mc$. The cyclotomic $q$-Schur algebra is the endomorphism algebra
\[\mathcal{S}=\End_{\h}
\left( \bigoplus_{\bmu \in \La^{r}} M^{\bmu}\right)\]
where each $M^{\bmu}$ is a certain $\h$-module. Then $\mathcal{S}$ is a cellular algebra in the sense of Graham and Lehrer~\cite{GL}. It is quasi-hereditary, with the cell modules and the simple modules both indexed by the $r$-multipartitions of $n$. If $\bla \in \Larn$, the cell module indexed by $\bla$ is called a Weyl module and is denoted $\Delta^{\bla}$ and the simple module indexed by $\bla$ is denoted $L^{\bla}$. For $\bla,\bmu \in \Larn$, let $[\Delta^{\bla}:L^{\bmu}]$ denote the multiplicity of $L^{\bmu}$ as a composition factor of $\Delta^{\bla}$.  

Work of Stroppel and Webster~\cite{SW} has shown that there is a grading on the cyclotomic $q$-Schur algebras analogous to that of the Ariki-Koike algebras; we define $[\Delta^{\bla}:L^{\bmu}]_v$ to be the graded multiplicity. Again, when $\mathcal{S}$ is defined over a field of characteristic $0$, the decomposition numbers are given by the transition coefficients for the Fock space. 

\begin{theorem} [\cite{SW}, Corollary~7.23] \label{T:WeylEqualI}
Suppose that $p=0$ and let $\bla,\bmu \in \Larn$. Then 
\[[\Delta^{\bla}:L^{\bmu}]_v = d_{\bla\bmu}(v).\]
\end{theorem}

Over a field of arbitrary characteristic, the decomposition numbers for the cyclotomic $q$-Schur algebras agree with those for the Ariki-Koike algebras. 

\begin{theorem} [\cite{Mathas:AKSurvey}, Theorem~5.1] \label{T:WeylEqualII}
Suppose that $\bla,\bmu \in \Larn$ with $\bmu \in \La^{\mc}_n$. Then 
\[[\Delta^{\bla}:L^{\bmu}] = [S^{\bla}:D^{\bmu}].\]
\end{theorem}

We define a partial order $\gedom$ on $\Larn$ by saying that $\bmu \gedom \bla$ if
\[\sum_{k=0}^{l-1} |\mu^{(k)}| + \sum_{i=1}^{t} \mu^{(l)}_i \geq \sum_{k=0}^{l-1} |\la^{(k)}| + \sum_{i=1}^{j} \la^{(l)}_i\]
for all $0 \leq l \leq r-1$ and all $t \geq 0$, and we write $\bmu \triangleright \bla$ if $\bmu \gedom \bla$ and $\bmu \neq \bla$. 
The next results follow from the construction of the cellular basis of $\mathcal{S}$~\cite[Theorem~4.11]{Mathas:AKSurvey} and the block structure of $\mathcal{S}$~\cite[Theorem~2.11]{LM:Blocks}.

\begin{proposition} \label{T:DomOrder}
Suppose that $\bla,\bmu \in \Larn$. 
\begin{itemize}
\item $[\Delta^{\bla}:L^{\bmu}]=1$ and $[\Delta^{\bla}:L^{\bmu}]=0$ unless $\bmu \gedom \bla$. 
\item $[\Delta^{\bla}:L^{\bmu}]=0$ unless $\bla \sim \bmu$. 
\end{itemize}
\end{proposition}

The next result is an application of a theorem of Bowman and Speyer. 

\begin{proposition} [\cite{BowSpey}, Corollary~3.15] \label{P:DiagCuts}
  Suppose that $r \geq 2$, that $\bla,\bmu \in \La^r_n$ and that $|\la^{(0)}|=|\mu^{(0)}|$. Let $\hat{\bmu}=(\mu^{(1)},\ldots,\mu^{(r-1)})$, $\hat{\bla}=(\la^{(1)},\ldots,\la^{(r-1)})$ and $\hat{\mc}=(a_1,\ldots,a_{r-1})$.
Consider the decomposition numbers $[\Delta^{(\la^{(0)})}:L^{(\mu^{(0)})}]$ and $[\Delta^{\hat{\bla}}:L^{\hat{\bmu}}]$ where the first is for a cyclotomic $q$-Schur algebra with multicharge $(a_0)$ and the second for a cyclotomic $q$-Schur algebra with multicharge $\hat{\mc}$. We have
  \[[\Delta^{\bla}:L^{\bmu}] = [\Delta^{(\la^{(0)})}:L^{(\mu^{(0)})}] [\Delta^{\hat{\bla}}:L^{\hat{\bmu}}].\]
\end{proposition}

Let $\mathscr{G}$ denote the Grothendieck group of $\mathcal{S}$ with $\Gr{M}$ denoting the isomorphism class of the $\mathcal{S}$-module $M$. By Proposition~\ref{T:DomOrder}, $\{\Gr{\Delta^{\bla}} \mid \bla \in \Larn\}$ is a basis of $\mathscr{G}$. Set $h_{\bla\bmu}=[\Delta^{\bla}:L^{\bmu}]$. 
For $\bmu \in \Larn$, let $P^{\bmu}$ denote the projective cover of the simple module $L^{\bmu}$. Then, by the properties of a cellular algebra,
\[\Gr{P^{\bmu}} = \sum_{\bla \in \Larn} h_{\bla\mu} \Gr{\Delta^{\bla}}.\]
Now let $\mathcal{S}^0$ denote a cyclotomic $q$-Schur algebra with the same parameters $r,n,e$ and $\mc$ as $\mathcal{S}$, but defined over a field of characteristic $0$. Let $h_{\bla\bmu}^0$ denote the multiplicity of the simple $\mathcal{S}^{0}$-module $L^{\bmu}$ as a composition factor of the $\mathcal{S}^0$-module $\Delta^{\bla}$. Let $D$ (resp. $D_0$) denote the decomposition matrix of $\mathcal{S}$ (resp. $\mathcal{S}^0$), that is the matrix whose rows and columns are indexed by the elements of $\Larn$ and where the $(\bla,\bmu)$-entry is $h_{\bla\bmu}$ (resp. $h^0_{\bla\bmu}$). We assume that the ordering on $D$ and $D_0$ is identical and is compatible with $\gedom$ so that by Proposition~\ref{T:DomOrder} the decomposition matrices are lower unitriangular.  

The reference we give below is actually for the $q$-Schur algebra, but the proof for the cyclotomic $q$-Schur algebra is identical. The matrix $A$ is known as an adjustment matrix. 

\begin{theorem} [\cite{Mathas:HeckeBook}, Theorem~6.35]\label{T:Adj}
There exists a square lower unitriangular matrix $A$ whose rows and columns are indexed by the elements of $\Larn$ and whose entries are non-negative integers such that
\[D = D_0 A.\]
Hence if $A=(a_{\btau\bmu})$ and $\bla,\bmu \in \Larn$, we have 
\[h_{\bla\bmu} = \sum_{\btau \in \Larn} h^0_{\bla\btau} a_{\btau\bmu}.\]
\end{theorem}

Let $i \in I$ and $m \geq 1$. As in~\cite{Wada}, we define an induction functor $\uparrow^{(m)}_i: \mathcal{S}_{r,n}\text{-mod} \rightarrow \mathcal{S}_{r,n+m}\text{-mod}$ with the property that if $\bnu \in \Larn$, 
\[\Gr{\Delta^{\bnu}} \uparrow^{(m)}_i = \sum_{\bnu \xrightarrow{m:i} \bla} \Gr{\Delta^{\bla}}.\] 
The reader should compare with the action of $f^{(m)}_i$ on $s_{\bnu}$ in Section~\ref{SS:Fock}.  
The functor $\uparrow^{(m)}_i$ is exact and so
\[\Gr{P^{\bnu}} \uparrow^{(m)}_i = \sum_{\bla \in \La^r_{n+m}} b_{\bla} \Gr{P^{\bla}},\]
for some $b_{\bla} \in \N$. 
Suppose that $1 \leq j \leq e-1$ and $s \geq 1$. Define 
\[\uparrow^{(s,j)} = \uparrow^{(s)}_{j} \circ  \uparrow^{(s)}_{2} \circ \uparrow^{(s)}_1 \circ \uparrow^{(s)}_{j+1}\circ \ldots \circ \uparrow^{(s)}_{e-1} \circ \uparrow^{(s)}_0\]
and compare the definition with that of $f^{(s,j)} \in \U$. Mirroring our notation in the previous section, if $\bmu \in \RR$, we set $Q(\bmu) = \sum_{\bla \approx \bmu} g_{\bla\bmu} \Gr{\Delta^{\bla}}$, where $g_{\bla\bmu}=g_{\bla\bmu}(1)$. Our set-up is then identical to that of Proposition~\ref{P:HeavyUngraded} and so we obtain the following result. 

\begin{lemma} \label{P:HeavyUngradedII}
Suppose $s >0$ and that $\bnu \in \Rrs$.
Let $1 \leq j \leq e-1$.
Then
  \[Q(\bnu) \uparrow^{(s)}_j= \sum_{\Delta \in \La}c^{\Delta}_{\nu^{0}_j (1^s)} Q((\nu^0_0,\ldots,\nu^{0}_{j-1},\Delta,\nu^{0}_{j+1},\ldots,\nu^0_{e-1}),(\nu^1_0,\nu^1_1,\ldots,\nu^1_{e-1}),\ldots,(\nu^{r-1}_0,\nu^{r-1}_1,\ldots,\nu^{r-1}_{e-1})). \]
  \end{lemma}

\begin{lemma} \label{L:ID}
  Suppose $s>0$ and that $\bnu \in \Rs$. Suppose that $\bsig \sim \bnu, \bsig \not \approx \bnu$ and $\bnu \gdom \bsig$. Suppose that $\bep \in \R$ is formed by adding $s$ $e$-rim hooks to the first component of $\bnu$. Suppose that for some $1 \leq j \leq e-1$, $\Gr{\Delta^{\btau}}$ appears in the sum $\Gr{\Delta^{\bsig}}\uparrow^{(s)}_j$ with non-zero coefficient. Then $\btau \not \gedom \bep$. 
\end{lemma}

\begin{proof}
Take all multipartitions as above and suppose that $\btau \gedom \bep$. Then \[|\tau^{(0)}| \geq |\epsilon^{(0)}| = |\nu^{(0)}|+se \geq |\sigma^{(0)}|+se \geq |\tau^{(0)}|,\]
so we must have equalities throughout the equation. 
Thus $\tau^{(k)}=\sigma^{(k)}$ and $\epsilon^{(k)}=\nu^{(k)}$ for all $k \geq 1$. However, thanks to the equality of the sizes of the first component and the dominance conditions, we have
\begin{align*}
&\sum_{k=1}^{l-1} |\nu^{(k)}|+ \sum_{i=1}^t \nu^{(l)}_i \geq \sum_{k=1}^{l-1} |\sigma^{(k)}|+ \sum_{i=1}^t \sigma^{(l)}_i, \\
 &\sum_{k=1}^{l-1} |\tau^{(k)}|+ \sum_{i=1}^t \tau^{(l)}_i \geq \sum_{k=1}^{l-1} |\epsilon^{(k)}|+ \sum_{i=1}^t \epsilon^{(l)}_i,
\end{align*}
for all $1\leq l \leq r-1$ and $t \geq 0$, 
so that in fact $\sigma^{(k)}=\nu^{(k)}$ for all $k \geq 1$. But $\bsig \sim \bnu$, so by Proposition~\ref{P:Naka}, we have $(\sigma^{(0)}) \approx (\nu^{(0)})$ so that $\bsig\approx \bnu$, giving the required contradiction. 
\end{proof}

\begin{lemma} \label{L:SameZero}
Suppose that $\bmu,\bla \in \R$ with $\bmu \approx \bla$ and $|\mu^0_i|=|\la^0_i|$ for all $0 \leq i \leq e-1$. Suppose that $\hook(\mu^{(0)})=\hook(\la^{(0)})<p$. Then $h_{\bla\bmu}=0$.   
\end{lemma}

\begin{proof}
Given the conditions on $\bmu$ and $\bla$, we have $|\mu^{(0)}|=|\lambda^{(0)}|$. By~\cite[Proposition~3.3]{JLM} we have that $[\Delta^{(\la^{(0)})}:L^{(\mu^{(0)})}]=0$. The lemma then follows from Proposition~\ref{P:DiagCuts}. 
\end{proof}

We are now ready to prove the main result of this section. Set $g_{\bla\bmu}=g_{\bla\bmu}(1)$. 

\begin{theorem} \label{T:MainSchur}
Suppose that $\bla\approx \bmu$ where $\bla \in \R$ and $\bmu \in \RR$. Suppose further that $p=0$ or $\hook(\mu^{(k)})<p$ for all $0 \leq k \leq r-1$. Then 
\[h_{\bla\bmu} = g_{\bla\bmu}.\]
\end{theorem}

\begin{proof}
If $p=0$ then the result follows from Theorem~\ref{T:WeylEqualI} and Theorem~\ref{T:Main}, so assume that $p>0$. 
As in the proof of Theorem~\ref{T:Main}, we use multiple induction arguments. The case $r=1$ appears as \cite[Theorem~4.1]{JLM}. So suppose that $r>1$, that Theorem~\ref{T:MainSchur} holds for $r-1$ and that $\bmu \in \RR$ with $\hook(\mu^{(k)})<p$ for $0 \leq k \le r-1$. Suppose first that $\hook(\mu^{(0)})=0$. 
For any $\bla = (\la^{(0)},\la^{(1)},\ldots,\la^{(r-1)}) \in \La^r$, set
$\hat{\bla}=(\la^{(1)},\ldots,\la^{(r-1}))$ and set $\hat{\mc}=(a_1,\ldots,a_{r-1})$.
Then $\hat{\bmu},(\mu^{(0)}) \in \RR$. 
Suppose that $\bla \approx \bmu$. If $\hook(\la^{(0)})>0$ then $\bmu \not \gedom \bla$ and $h_{\bla\mu}=g_{\bla\bmu}=0$ so assume $\hook(\la^{(0)})=0$, that is, $\la^{(0)}=\mu^{(0)}$. Now, applying Proposition~\ref{P:DiagCuts} we have
\[h_{\bla\bmu}  =  h_{(\la^{(0)})(\mu^{(0)})} h_{\hat{\bla}\hat{\bmu}}= g_{\hat{\bla}\hat{\bmu}} =g_{\bla\bmu}\]
where the middle step follows from the inductive hypothesis and Proposition~\ref{T:DomOrder} and the last step follows from Lemma~\ref{L:MinusOK}. Hence Theorem~\ref{T:MainSchur} holds when $\hook(\mu^{(0)})=0$.

Now suppose that $p>\hook(\mu^{(0)})>0$ and that Theorem~\ref{T:MainSchur} holds for all $\bnu \in \RR$ with $\hook(\nu^{(0)})<\hook(\mu^{(0)})$ and $\hook(\nu^{(k)})<p$ for $1 \leq k \leq r-1$. 
Choose $1 \leq j \leq e-1$ such that $\mu^0_j \neq \emp$. Let $\nu$ be the partition whose Young diagram is obtained by removing the first column from $[\mu^0_j]$, where we suppose that we remove $s\geq 1$ nodes. Let $\bnu$ be the multipartition with the same multicore as $\bmu$ and
\[\nu^k_i = \begin{cases} \nu, & k=0 \text{ and } i=j, \\
  \mu^k_i, & \text{otherwise}.\end{cases}\]
By the inductive hypothesis,
\[\Gr{P^{\bnu}}= \sum_{\bsig \approx \bnu} g_{\bsig\bnu} \Gr{\Delta^{\bsig}} +
\sum_{\substack{\bsig \sim \bnu \\ \bsig \not \approx \bnu}} h_{\bsig\bnu} \Gr{\Delta^{\bsig}} = Q(\bnu) + \sum_{\substack{\bsig \sim \bnu \\ \bsig \not \approx \bnu}} h_{\bsig\bnu} \Gr{\Delta^{\bsig}}.\]
Now by Lemma~\ref{P:HeavyUngradedII}, 
\[\Gr{P^{\bnu}} \uparrow^{(s)}_j = Q(\bnu) \uparrow^{(s)}_j +
\sum_{\substack{\bsig \sim \bnu \\ \bsig \not \approx \bnu}} h_{\bsig\bnu} \Gr{\Delta^{\bsig}} \uparrow^{(s)}_j
= \sum_{\Delta}c^{\Delta}_{\nu (1^s)} Q({\bf \Delta}) + \sum_{\substack{\bsig \sim \bnu \\ \bsig \not \approx \bnu}} h_{\bsig\bnu} \Gr{\Delta^{\bsig}}\uparrow^{(s)}_j;
\]
we also have
\[\Gr{P^{\bnu}} \uparrow^{(s)}_j = \sum_{\bep \approx \bmu} b_{\bep} \Gr{\Delta^{\bep}} + \sum_{\substack{\bep \sim \bmu \\ \bep \not \approx \bmu}} b_{\bep} \Gr{\Delta^{\bep}}
\]
for some $b_{\bep} \in \N$. 
Take $\Delta \neq \mu^{0}_j$ such that $c^{\Delta}_{\nu(1^s)} \neq 0$ and define $\bDel$ as in the proof of Proposition~\ref{P:HeavyGraded}. By Lemma~\ref{L:SameZero} we have $h_{\bDel\bmu}=0$.
Take $\bep \sim \bmu$ such that $b_{\bep} \neq 0$. 
\begin{itemize}
\item If $\bep \not \approx \bmu$ then by Lemma~\ref{L:ID}, $\bep \not \gedom \bDel$ and so $h_{\bDel\bep}=0$. 
\item If $\bep \approx \bmu$ and $\bep \gdom \bDel$ then $\ep^k_i=\mu^k_i$ unless $k=0$ and $i=j$. Hence by Lemma~\ref{L:SameZero}, $h_{\bDel\bep}=0$.  
\end{itemize}
Thus the only way that we have $h_{\bDel\bmu}=0$ is if  
\[\Gr{P^{\bnu}} \uparrow^{(s)}_j
= \sum_{\Delta} c^{\Delta}_{\nu(1^s)} \Gr{P^{\bDel}} + \sum_{\bep} r_{\bep} \Gr{P^{\bep}}\]
for some $r_{\bep} \in \N$. Consider $\bla \approx \bmu$. Then the coefficient of $\Gr{\Delta^{\bla}}$ in $\Gr{P^{\bnu}} \uparrow^{(s)}_j$ is
\[\sum_{\Delta} c^{\Delta}_{\nu(1^s)} g_{\bla \bDel} = \sum_{\Delta}c^{\Delta}_{\nu(1^s)} h_{\bla\bDel} + \sum_{\bep}r_{\bep} h_{\bla \bep}.\]
But by Theorem~\ref{T:Adj}, each $\bDel$ that appears in the sum above is such that $g_{\bla \bDel} = h^0_{\bla\bDel} \leq h_{\bla\bDel}$, and so we must have $g_{\bla \bDel}=h_{\bla \bDel}$. Taking $\bDel=\bmu$, we complete the proof of the theorem.   
\end{proof}

\begin{theorem} \label{T:MainP}
Suppose that $\bla\approx \bmu$ where $\bla,\bmu \in \R$ and $\bmu \in \La^{\mc}$. Suppose further that $\hook(\mu^{(k)})<p$ for all $0 \leq k <r$. Then 
\[[S^{\bla}:D^{\bmu}]_v = g_{\bla\bmu}(v).\]
\end{theorem}

\begin{proof}
The ungraded version of Theorem~\ref{T:MainP} follows from Theorem~\ref{T:MainSchur} and Theorem~\ref{T:WeylEqualI}. Section~10.3. The graded version follows from Theorem~\ref{T:Main2} because there is a graded adjustment matrix that relates the graded decomposition numbers in characteristic $0$ and characteristic $p$~\cite[Section~10.3]{Kleshchev:Survey}. 
\end{proof}

\subsection{Scopes equivalences} \label{SS:Scopes}
A celebrated paper of Scopes~\cite{Scopes} proves certain equivalences between blocks of the symmetric group algebra. Scopes' paper, which was generalized by Jost~\cite{Jost} to the Hecke algebras of type $A$, shows that the blocks are Morita equivalent and that there is a corresponding bijection between the partitions in the respective blocks which preserves the decomposition matrices. As we explain below, an generalization of the decomposition number result to the Ariki-Koike algebras was recently given by Dell'Arciprete~\cite{DA} and a generalization of the  Morita equivalence is a special case of even more recent work by Webster~\cite{Webster}. 

Fix a multicharge $\mc \in I^r$. For $0 \leq i \leq e-1$, define $\phi_i: \Z \rightarrow \Z$ by setting
\[\phi_i(b) = \begin{cases} b+1, & b \equiv i-1 \mod e \\
b-1, & b \equiv i \mod e, \\
0, & \text{otherwise}.
\end{cases}\]
Suppose $\bla \in \La^r$. Define $\Phi_i(\bla)$ to be the multipartition where the $\beta$-set of component $k$ is equal to $\Phi_i(\mathfrak{B}^k_{a_k}(\bla))$, that is we obtain the abacus configuration of $\Phi_i(\bla)$ from that of $\bla$ by swapping runners $i$ and $i+1$ on all components (with a vertical shift if $i=0$).  

The map $\Phi_i$ preserves $\sim_{\mc}$-equivalence classes.

\begin{lemma}  [\cite{Fayers:Weights}, Proposition~4.6] \label{L:BlockScopes}
Let $\bla,\bmu \in \La^r$. Then $\bla \sim_{\mc} \bmu$ if and only if $\Phi_i(\bla) \sim_{\mc} \Phi_i(\bmu)$. 
\end{lemma}

Let $B$ be a $\sim_{\mc}$-equivalence class of $\La^r$. Say that the $\sim_{\mc}$-equivalence class $\tilde{B}$ is formed from $B$ by making a Scopes move if $\tilde{B}=\Phi_i(B)$ for some $0 \leq i \leq e-1$ and no multipartition $\bla \in B$ has any addable $i$-nodes. Let $\equiv_{\text{Sc}}$ be the equivalence relation on the $\sim_{\mc}$-equivalence classes of $\La^r$ generated by making Scopes moves; we call this Scopes equivalence. If $B \equiv_{\text{Sc}} \tilde{B}$ then by composing the bijections $\Phi_i, \Phi^{-1}_i$ which give the Scopes moves between $B$ and $\tilde{B}$, we have a bijection $\Phi: B \rightarrow \tilde{B}$. 

\begin{proposition} [\cite{DA}, Proposition~5.5] \label{P:DeEq}
Suppose that $B$ and $\tilde{B}$ are $\sim_{\mc}$-equivalence classes with $B \equiv_{\text{Sc}} \tilde{B}$. Suppose $\bla,\bmu \in B$ with $\bmu$ a Kleshchev multipartition. Then $\Phi(\bmu)$ is a Kleshchev multipartition and we have $[S^{\bla}:D^{\bmu}]=[S^{\Phi(\bla)}:D^{\Phi(\bmu)}]$. 
\end{proposition}

\begin{proposition} [\cite{Webster}, Lemma~3.2]
Suppose that $B$ and $\tilde{B}$ are $\sim_{\mc}$-equivalence classes with $B \equiv_{\text{Sc}} \tilde{B}$. Then the blocks of the Ariki-Koike algebras corresponding to $B$ and $\tilde{B}$ are Morita equivalent. 
\end{proposition} 

Using Webster's work and the results of~\cite[Section~3.3]{L:Rouquier}, we can describe when a block is Scopes equivalent to a Rouquier block. 
Let $\bla \in \La^r$ and recall the definition of $\mathfrak{b}^k_i(\bla)$ from Subsection~\ref{SS:Rouq}. For $0 \leq i \leq e-1$, define $\mathfrak{b}^{\ast}_i(\bla) = \sum_{k=0}^{r-1}\mathfrak{b}^k_i(\bla)$. By~\cite[Lemma~3.2]{Fayers:Weights}, the function $\mathfrak{b}^\ast_i$ is constant on $\sim_{\mc}$-equivalence classes, so if $B$ is such a class, we can define $\mathfrak{b}^{\ast}_i(B)$. Define a total order $\precdot$ on $\{0,1,\ldots,e-1\}$ by saying that
\[i \precdot j \text{ if } \mathfrak{b}^{\ast}_i(B) < \mathfrak{b}^{\ast}_j(B) \text{ or if } \mathfrak{b}^{\ast}_i(B) = \mathfrak{b}^{\ast}_j(B) \text{ and } i<j\]
and define $\pi=\pi(B)$ to be the permutation such that 
\[\mathfrak{b}^{\ast}_{\pi(0)}(B) \precdot \mathfrak{b}^{\ast}_{\pi(1)}(B) \precdot \ldots \precdot \mathfrak{b}^{\ast}_{\pi(e-1)}(B).\] 
We say that a $\sim_{\mc}$-equivalence class $B$ is a RoCK block if every $\bla \in B$ satifies
\[\hook(\bla) \leq \mathfrak{b}^{k}_{\pi(i)}(\bla) - \mathfrak{b}^{k}_{\pi(i-1)}(\bla)+1\]
for all $1 \leq i \leq e-1$ and $0 \leq k \leq r-1$. 

\begin{proposition} [\cite{Webster}, Proposition~4.3 \& \cite{L:Rouquier}, Section~3.3]
A $\sim_{\mc}$-equivalence class $B$ is a RoCK block if and only if it is Scopes equivalent to a Rouquier block. 
\end{proposition}

We could equally have defined a RoCK block to be a block which is Scopes equivalent to a Rouquier block and then given the equivalent combinatorial definition. Our terminology follows that of~\cite{Webster}, although Webster's RoCK blocks are defined more generally; when applied to the Ariki-Koike algebras, the notations coincide. Applying Proposition~\ref{P:DeEq}, we are now in a position to give some decomposition numbers for RoCK blocks. 

\begin{theorem}
Suppose that $\h_{r,n}(q,{\bf Q})$ is defined over a field of characteristic $p \geq 0$. Take $\bla \in \La^r_n$ and $\bmu \in \La^{\mc}_n$ such that $\bla$ and $\bmu$ lie in a RoCK block $\mathcal{R}$ and $\bla \approx_{\mc} \bmu$. Let $\pi=\pi(\mathcal{R})$ be the permutation defined above. Suppose that $p=0$ or $\hook(\mu^{(k)})<p$ for all $0 \leq k \le r-1$. Then
\[
[S^{\bla}:D^{\bmu}] = \sum_{\bal \in \M{r}{e+1}} \sum_{\bbe \in \M{r}{e}} \sum_{\bgam \in \M{{r+1}}{e}} \sum_{\bdel \in \M{r}{e}} \left( \prod_{k=0}^{r-1} \prod_{i=0}^{e-1} c^{\delta^k_i}_{\mu^k_{\pi(i)} \gamma^k_i} c^{\delta^k_i}_{ \al^k_i \be^k_i \gamma^{k+1}_i} c^{\la^k_{\pi(i)}}_{\be^k_i (\al^k_{i+1})'}\right)  c^{\emp}_{\gamma^0_0 \gamma^0_1 \ldots \gamma^{0}_{e-1}}  c^{\emp}_{\gamma^r_0 \gamma^r_1 \ldots \gamma^r_{e-1}}.
\]
\end{theorem}  

\subsection{Open questions} \label{SS:Open}

In Theorem~\ref{T:Main}, we show that $d_{\bla\bmu}(v)=g_{\bla\bmu}(v)$ for any $(\bla,\bmu) \in \Rp$, and in order to obtain Theorem~\ref{T:Main2} we just ignore any $\bmu \in \RR$ not indexed by a Kleshchev multipartition. The definition of $g_{\bla\bmu}(v)$ does not depend on the common multicore of $\bla$ and $\bmu$. However, the set $\La^{\mc} \cap \RR$ does. Unfortunately we do not have a non-recursive way of testing whether a multipartition in a Rouquier block is a Kleshchev multipartition; that is, we would like an analogue of Lemma~\ref{L:Rouqreg} for $r >1$. 

\begin{ex} Let $e=2$ and $r=2$ and let 
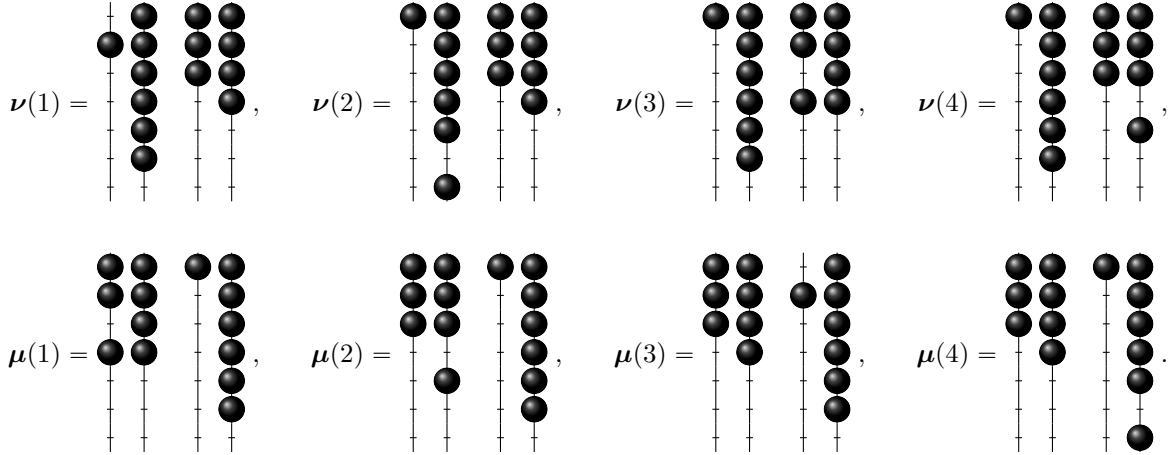
\begin{align*}
\bnu(1) & = \abacus(nb,bb,nb,nb,nb,nb,nn) \quad
\abacus(bb,bb,bb,nb,nn,nn,nn)\;, &
\bnu(2) & = \abacus(bb,nb,nb,nb,nb,nn,nb) \quad 
\abacus(bb,bb,bb,nb,nn,nn,nn)\;, &
\bnu(3) & = \abacus(bb,nb,nb,nb,nb,nb,nn) \quad
\abacus(bb,bb,nb,bb,nn,nn,nn)\;, &
\bnu(4) & = \abacus(bb,nb,nb,nb,nb,nb,nn) \quad 
\abacus(bb,bb,bb,nn,nb,nn,nn)\;, \\ \\
\bmu(1) & = \abacus(bb,bb,nb,bb,nn,nn,nn) \quad 
\abacus(bb,nb,nb,nb,nb,nb,nn)\;, &  
\bmu(2) & = \abacus(bb,bb,bb,nn,nb,nn,nn) \quad 
\abacus(bb,nb,nb,nb,nb,nb,nn)\;, &  
\bmu(3) & = \abacus(bb,bb,bb,nb,nn,nn,nn) \quad 
\abacus(nb,bb,nb,nb,nb,nb,nn)\;, &  
\bmu(4) & = \abacus(bb,bb,bb,nb,nn,nn,nn)  \quad
\abacus(bb,nb,nb,nb,nb,nn,nb)\;.
\end{align*}
The $2$-regular multipartitions are $\bnu(2), \bnu(4)$, $\bmu(2)$ and $\bmu(4)$, and $d_{\bnu(x)\bnu(y)}(v) = d_{\bmu(x)\bmu(y)}(v)$ for all $1 \leq x \leq 4$ and $y=2,4$. However 
$\bnu(2)$ and $\bnu(4)$ are Kleshchev multipartitions whereas $\bmu(2)$ and $\bmu(4)$ are not. 
\end{ex} 

If $r=1$, we are able to express the decomposition numbers $[S^{\bla}:D^{\bmu}]_v$ in terms of other (unknown) decomposition numbers.

\begin{proposition}  [\cite{JLM}, Proposition~4.3] \label{T:AdjP}
Let $r=1$. Suppose that $\mu \in \RR$ and $\bla \approx \bmu$. Set 
\[T(\bmu) = \{\btau \approx \bmu: |\tau^k_i|=|\mu^k_i| \text{ for all } 0 \le i \leq e-1 \text{ and } 0 \leq k \leq r-1\}.\]
Then 
\[h_{\bla\bmu} = \sum_{\btau \in T(\bmu)} g_{\bla\btau} h_{\btau\bmu}.\]
\end{proposition}

We were initially hopeful that an analogue of this result held for $r\geq 2$. We do not have any examples that contradict it, however we do not think it is likely to hold. When $\bnu \not \approx \bmu$, we have no control over the entries $a_{\bnu\bmu}$ of the adjustment matrix.  

In Conjecture~\ref{C:Schur}, we conjectured that we have a formula for the decomposition numbers $[\Delta(\bla):L(\bmu)]$ for the cyclotomic $q$-Schur algebras where $\bla \approx \bmu$ lie in a Rouquier block that holds for arbitrary $\bmu$, rather than $\bmu$ $e$-regular as in Theorem~\ref{T:MainSchur}. The formula differs from $g_{\bla\bmu}(1)$ only by the addition of a term $c^{\emp}_{\alpha^0_0 \al^1_0 \ldots \al^{r-1}_{0}}$. In~\cite[Corollary~3.12]{JLM}, we proved Conjecture~\ref{C:Schur} in the case that $r=1$ using a runner-removal result of James and Mathas~\cite{JM:Runner}. Unfortunately we do not have an analogue of the runner removal theorem for $r>1$.

\end{document}